\documentclass{amsart}
\author{J{\"o}rg Brendle}
\thanks{First author partially supported by Grants-in-Aid for Scientific Research (C) 21540128
and (C) 24540126, Japan Society for the Promotion of Science, and by the Fields Institute for Research in Mathematical Sciences, Toronto, Canada}
\email{brendle@kurt.scitec.kobe-u.ac.jp}
\author{Dilip Raghavan}
\thanks{Second author partially supported by Grant-in-Aid for Scientific Research for JSPS
Fellows No.\ 23$\cdot$01017}
\address{Department of Mathematics \\
National University of Singapore\\
Singapore 119076}
\email{raghavan@math.nus.edu.sg}
\date{\today}
\subjclass[2010]{03E05, 03E17, 03E35, 03E65}
\keywords{maximal almost disjoint family, cardinal invariants}
\title{Bounding, splitting, and almost disjointness}
\usepackage{amssymb, amsmath, amsthm, mathrsfs, enumerate, amsfonts, latexsym, bbm}
\def\polhk#1{\setbox0=\hbox{#1}{\ooalign{\hidewidth
    \lower1.5ex\hbox{`}\hidewidth\crcr\unhbox0}}}
\newtheorem{Theorem}{Theorem}

\newtheorem{Lemma}[Theorem]{Lemma}
\newtheorem{Cor}[Theorem]{Corollary}
\newtheorem{Obs}[Theorem]{Observation}

\newtheorem{Question}[Theorem]{Question}
\theoremstyle{definition}
\newtheorem{Def}[Theorem]{Definition}

\theoremstyle{remark}

\newtheorem{ssclaim}{Claim}[Theorem]

\newcommand{\forces}{\Vdash}
\newcommand{\restrict}{\upharpoonright}
\newcommand{\re}{{\upharpoonright}}
\newcommand{\forallbutfin}{{\forall}^{\infty}}
\newcommand{\existsinf}{{\exists}^{\infty}}
\renewcommand{\c}{\mathfrak{c}}
\renewcommand{\b}{\mathfrak{b}}

\newcommand{\s}{\mathfrak{s}}

\renewcommand{\a}{{\mathfrak{a}}}
\newcommand{\ac}{{\mathfrak{a}}_{closed}}
\renewcommand{\[}{\left[}
\renewcommand{\]}{\right]}
\renewcommand{\P}{\mathbb{P}}

\newcommand{\Q}{\mathbb{Q}}
\newcommand{\lc}{\left|}
\newcommand{\rc}{\right|}
\newcommand{\<}{\prec}

\newcommand\FIN{\mathrm{FIN}}

\newcommand\CH{\mathrm{CH}}
   
\newcommand{\BS}{{\omega}^{\omega}}

\DeclareMathOperator{\nor}{nor}
\DeclareMathOperator{\interior}{int}

\DeclareMathOperator{\dom}{dom}

\DeclareMathOperator{\succc}{succ}

\newcommand{\Pset}{\mathcal{P}}

\newcommand{\Mi}{\mathbb{M}}

\newcommand{\A}{{\mathscr{A}}}

\newcommand{\CC}{{\mathbb{C}}}

\newcommand{\GG}{{\mathcal{G}}}
\newcommand{\E}{{\mathcal{E}}}
\newcommand{\U}{{\mathcal{U}}}

\newcommand{\cube}{{\[\omega\]}^{\omega}}
\newcommand{\fin}{{{\[\omega\]}^{< \omega}}}
\newcommand{\I}{{\mathcal{I}}}
\newcommand{\F}{{\mathcal{F}}}
\newcommand{\FF}{{\mathbb{F}}}
\newcommand{\T}{{\mathcal{T}}}

\newcommand{\V}{{\mathbf{V}}}
\newcommand{\LLL}{{\mathbf{L}}}

\newcommand{\R}{\mathbb{R}}

\renewcommand{\aa}{{\mathfrak a}}

\newcommand{\bb}{{\mathfrak b}}

\newcommand{\cc}{{\mathfrak c}}
\newcommand{\dd}{{\mathfrak d}}

\newcommand{\hh}{{\mathfrak h}}

\newcommand{\rr}{{\mathfrak r}}

\renewcommand{\ss}{{\mathfrak s}}
\newcommand{\sss}{{\mathfrak s}}

\newcommand{\VVV}{{\mathbf V}}
\newcommand{\LL}{{\mathbb L}}
\newcommand{\MM}{{\mathbb M}}
\newcommand{\HHH}{{\mathcal H}}

\newcommand{\sub}{\subseteq}
\newcommand{\sem}{\setminus}
\newcommand{\twoom}{2^\omega}

\newcommand{\omlom}{\omega^{<\omega}}
\newcommand{\omom}{\omega^\omega}
\newcommand{\omloms}{[\omega]^{<\omega}}

\newcommand{\omoms}{[\omega]^\omega}

\newcommand{\ha}{\,{}\hat{}\,}
\newcommand{\la}{\langle}
\newcommand{\ra}{\rangle}
\newcommand{\closed}{{\mathrm{closed}}}
\newcommand{\Borel}{{\mathrm{Borel}}}
\newcommand{\rk}{{\mathrm{rk}}}
\newcommand{\stem}{{\mathrm{stem}}}

\begin{document}
\begin{abstract}
We investigate some aspects of bounding, splitting, and almost disjointness.
In particular, we investigate the relationship between the bounding number, the
closed almost disjointness number, splitting number, and the existence of
certain kinds of splitting families.
\end{abstract}
\maketitle

\section{Introduction} \label{sec:intro}
	The closed (and Borel) almost disjointness number was recently
introduced by Brendle and Khomskii~\cite{BKta}, and has received a lot of
attention.
We study the connections between this number and the notions of bounding and
splitting in this paper.
We start with some basic definitions. 
Recall that two infinite subsets $a$ and $b$ of $\omega$ are \emph{almost
disjoint or a.d.\@} if $a \cap b$ is finite.
We say that a family $\A$ of infinite subsets of $\omega$ is \emph{almost
disjoint or a.d.\@} if its members are pairwise almost disjoint.
A \emph{Maximal Almost Disjoint family, or MAD family} is an infinite a.d.\ family with the property that $\forall b \in \cube \exists a \in \A \[\lc a \cap b \rc = {\aleph}_{0}\]$.
The cardinal invariant $\a$ is the least $\kappa$ such that there is a MAD family of size $\kappa$.
Recall that $\b$ is the least size of a subset of $\langle \BS, {\leq}^{\ast} \rangle$ that does not have an upper bound.
It is well-known that $\b \leq \a$.
For $x, a \in \Pset(\omega)$, $x$ \emph{splits} $a$ if $\lc x \cap a \rc = \lc
\left( \omega \setminus x \right) \cap a \rc = \omega$.
$\F \subset \Pset(\omega)$ is called a \emph{splitting family} if $\forall a \in
\cube \exists x \in \F\[x \ \text{splits} \ a\]$.
$\s$ is the least size of a splitting family.
$\F \subset \Pset(\omega)$ is called an \emph{$\omega$-splitting family} if for
any collection $\{{a}_{n}: n \in \omega\} \subset \cube$, there exists $x \in
\F$ such that $\forall n \in \omega \[x \ \text{splits} \ {a}_{n}\]$.
${\s}_{\omega}$ is the least size of an $\omega$-splitting family.

Brendle and Khomskii~\cite{BKta} studied the possible descriptive
complexities of MAD families in certain forcing extensions of $\LLL$.
This led them to consider the following cardinal invariant.
\begin{Def} \label{def:ac}
	$\ac$ is the least $\kappa$ such that there are $\kappa$ closed subsets
of $\cube$ whose union is a MAD family in $\cube$. 
\end{Def}
Obviously, $\ac \leq \a$. Brendle and Khomskii showed in \cite{BKta} that $\ac$
behaves differently from $\a$ by showing that $\ac = {\aleph}_{1} < {\aleph}_{2}
= \b$ holds in the Hechler model.
Heuristically, the difference between $\a$ and $\ac$ may be seen by considering
how a witness to $\ac = {\aleph}_{1}$ can be destroyed in a forcing extension.
If $\A = {\bigcup}_{\alpha < {\omega}_{1}}{{X}_{\alpha}}$ is a witness to $\ac =
{\omega}_{1}$, where the ${X}_{\alpha}$ are closed subsets of $\cube$ coded in
the ground model, then to destroy $\A$ it is necessary to add a set $b \in
\cube$ which is almost disjoint from every member of every ${X}_{\alpha}$
\emph{even after} these codes have been reinterpreted in the forcing extension.
Interpreting a ground model code in a forcing extension results in a larger set
of reals.
This makes increasing $\ac$ harder than increasing $\a$, and this fact was
exploited by Brendle and Khomskii in their above mentioned result.
  
In Sections \ref{sec:blessthanaclosed} and \ref{sec:cccproof} we prove the
consistency of $\b < \ac$.
So taken together with the earlier result of Brendle and Khomskii, this
establishes the mutual independence of $\b$ and $\ac$.
Unsurprisingly, our proofs are closely modeled on the existing proofs of the
consistency of $\b < \a$. 
Historically there have been two seemingly distinct methods for producing a
model of $\b < \a$. 
In the first method, invented by Shelah in \cite{b<s}, the conditions consist of
a finite part followed by an infinite sequence of finite sets equipped with a
measure-like structure.
In the same paper, Shelah also used this method to produce the first consistency
proof of $\b < \s$.
In Section \ref{sec:blessthanaclosed}, we get a model of $\b < \ac$ using
Shelah's technique.  
In the second method, devised by Brendle in \cite{Br98}, an ultrafilter is
constructed as an ascending union of ${F}_{\sigma}$ filters, and then this
ultrafilter is diagonalized by the corresponding Mathias-Prikry forcing. 
One of the byproducts of the results in this paper is that these two techniques
are not so different after all. 
In Section \ref{sec:twostep} we show that Shelah's forcing from \cite{b<s} is
equivalent to a two step iteration of a countably closed forcing that adds an
ultrafilter which is a union of ${F}_{\sigma}$ filters from the ground model
succeeded by the Mathias-Prikry forcing for this generic ultrafilter.
Examining this proof one quickly realizes that for the Mathias-Prikry forcing
occurring in the second step of this iteration to have the right properties, it
is not necessary for the ultrafilter to be fully generic with respect to the
countably closed forcing occurring in the first step; it is sufficient for the
ultrafilter to meet a certain collection of $\c$ many dense sets.
With this realization, assuming $\CH$, it is possible to build a sufficiently
generic ultrafilter in the ground model itself.
In this way, we give a proof of the consistency of $\b < \ac$ by a finite
support iteration of Mathias-Prikry forcings in Section \ref{sec:cccproof} along
the lines of Brendle~\cite{Br98}.

In Section \ref{sec:club} we show that the existence of certain special types of
splitting families implies that $\ac = {\omega}_{1}$.
The existence of such special splitting families is closely related to the
statement ${\s}_{\omega} = {\omega}_{1}$.
It is unknown whether ${\s}_{\omega} = {\omega}_{1}$ implies that $\ac
= {\omega}_{1}$.
The result in Section \ref{sec:club} sheds some light on this, and moreover it
strengthens previous results of Raghavan and Shelah~\cite{RS12}, and Brendle and
Khomskii~\cite{BKta}. 

Finally in Section \ref{sec:clubandtail}, we separate the notions of club
splitting and tail splitting (see Definition \ref{def:clubandtail}). This
answers a question from \cite{GSta2}.     
\section{Consistency of ${\aleph}_{1} = \b < \ac$} \label{sec:blessthanaclosed}
In this section we show the consistency of $\b < \ac$ by a creature forcing.
The argument is similar to the one used by Shelah in \cite{b<s} and \cite{PIF}
to show the consistency of $\b < \a$, though we have to do some extra work to
make this argument work for $\ac$.
The notation and presentation in this section generally follow Abraham~\cite{Ab}.

Before plunging into the details, we make some remarks about the structure of the proof.
The final forcing will be a countable support (CS) iteration of proper forcings which does not add a dominating real.
At any stage, a specific witness to $\ac = {\omega}_{1}$, call it $\A$, is dealt with.
We first define a proper poset ${\P}_{0}$ which adds an unsplit real but does not add any dominating reals (and more; see Definition \ref{def:almost} and
following discussion).
The definition of ${\P}_{0}$ does not depend on $\A$, and it may or may not
destroy $\A$.
If it does, then we simply force with ${\P}_{0}$.
If it does not, we first add ${\omega}_{1}$ Cohen reals.
In the resulting extension we define a proper poset ${\P}_{1}$ which depends on
$\A$ and always destroys it.
Under the assumption that ${\P}_{0}$ (as defined in the extension) still does
not destroy $\A$, we prove that ${\P}_{1}$ does not add dominating reals (and
more), so that we may force with ${\P}_{1}$ to take care of $\A$.
	\begin{Def} \label{def:norm}
		$\FIN$ denotes $\fin \setminus \{ 0 \}$. Let $x \subset \omega$.
A function $\nor: {\[x\]}^{< \omega} \rightarrow \omega$ is a said to be a
\emph{norm on $x$} if
			\begin{enumerate}
				\item
					$\forall s \in {\[x\]}^{< \omega}
\[\nor(s) > 0 \implies \lc s \rc > 1\]$
				\item
					$\forall s, t \in {\[x\]}^{< \omega} \[s
\subset t \implies \nor(s) \leq \nor(t)\]$
				\item
					for any $s, {s}_{0}, {s}_{1} \in
{\[x\]}^{< \omega}$ and for any $n > 0$, if $\nor(s) \geq n$ and $s = {s}_{0}
\cup {s}_{1}$, then there exists $i \in 2$ such that $\nor({s}_{i}) \geq n - 1$.
			\end{enumerate}
	A \emph{creature} $c$ is a pair $\langle {s}_{c}, {\nor}_{c} \rangle$
such that ${s}_{c} \in \FIN$ and ${\nor}_{c}$ is a norm on ${s}_{c}$ such that
${\nor}_{c}({s}_{c}) > 0$. Given creatures $c$ and $d$, we write $c < d$ to mean
$\max({s}_{c}) < \min({s}_{d})$ and ${\nor}_{c}({s}_{c}) < {\nor}_{d}({s}_{d})$.

	A \emph{$0$-condition} $p$ is a pair $\langle {s}^{p}, \langle
{c}^{p}_{n}: n \in \omega \rangle \rangle$ such that
		\begin{enumerate}
			\item[(4)]
				${s}^{p} \in \fin$
			\item[(5)]
				for each $n \in \omega$, ${c}^{p}_{n}$ is a
creature and ${c}^{p}_{n} < {c}^{p}_{n + 1}$
			\item[(6)]
				$\forall m \in {s}^{p} \[m <
\min\left({s}_{{c}^{p}_{0}}\right)\]$
		\end{enumerate}	
	Henceforth, ${s}^{p}_{n}$ and ${\nor}^{p}_{n}$ will be used to denote
${s}_{{c}^{p}_{n}}$ and ${\nor}_{{c}^{p}_{n}}$ respectively. We may also omit
the superscript $p$ if it is clear from the context. For a $0$-condition $p$,
$\interior(p) = {\bigcup}_{n \in \omega}{{s}^{p}_{n}}$. Given $0$-conditions $p$
and $q$, $q \leq p$ means
		\begin{enumerate}
			\item[(7)]
				${s}^{q} \supset {s}^{p}$ and ${s}^{q} \setminus
{s}^{p} \subset \interior(p)$ 
			\item[(8)]
				Let ${n}_{0}$ be least such that $\forall m \geq
{n}_{0} \[{s}^{p}_{m} \cap {s}^{q} = 0\]$. There exists an interval partition
$\langle {i}_{n}: n \in \omega \rangle$ of $[{n}_{0}, \infty)$ (that is,
${i}_{0} = {n}_{0}$ and $\forall n \in \omega \[{i}_{n} < {i}_{n + 1}\]$) such
that $\forall n \in \omega \[{s}^{q}_{n} \subset {\bigcup}_{m \in [{i}_{n},
{i}_{n + 1})}{{s}^{p}_{m}}\]$.
			\item[(9)]
				for any $n \in \omega$, for any $t \subset
{s}^{q}_{n}$, if ${\nor}^{q}_{n}(t) > 0$, then there is $m \in \omega$ such that
${\nor}^{p}_{m}(t \cap {s}^{p}_{m}) > 0$. 
		\end{enumerate} 	
		For $0$-conditions $p$ and $q$, we say $q {\leq}_{0} p$ if $q
\leq p$ and ${s}^{p} = {s}^{q}$. For $n > 0$, $q {\leq}_{n} p$ if $q {\leq}_{0}
p$ and for all $m \leq n - 1$, ${c}^{q}_{m} = {c}^{p}_{m}$.  	
	\end{Def}
Observe that clause (8) is equivalent to saying that for each $n \in \omega$, ${s}^{q}_{n} \subset {\bigcup}_{m \in [{n}_{0}, \infty)}{{s}^{p}_{m}}$ and $\max\{m \in [{n}_{0}, \infty): {s}^{q}_{n} \cap {s}^{p}_{m} \neq 0 \} < \min\{m \in [{n}_{0}, \infty): {s}^{q}_{n + 1} \cap {s}^{p}_{m} \neq 0\}$. This is sometimes useful for checking clause (8). Also, it is easy to see that $\leq$
and ${\leq}_{n}$ are transitive for all $n$.
\begin{Lemma} \label{lem:fusion}
	Let $\langle {p}_{n}: n \in \omega \rangle$ be a sequence of
$0$-conditions and let $\langle {k}_{n}: n \in \omega\rangle$ be a sequence of
elements of $\omega \setminus \{0\}$ such that $\forall n \in \omega \[{k}_{n} <
{k}_{n + 1}\]$. Assume that ${p}_{n + 1} \; {\leq}_{{k}_{n}} \; {p}_{n}$. Define
$q$ as follows. ${s}^{q} = {s}^{{p}_{n}}$ for all $n$. For all $m \in [0,
{k}_{0})$, ${c}^{q}_{m} = {c}^{{p}_{0}}_{m}$. For each $m \in [{k}_{n}, {k}_{n +
1})$, ${c}^{q}_{m} = {c}^{{p}_{n + 1}}_{m}$. Then $q$ is a $0$-condition and for
each $n \in \omega$, $q \; {\leq}_{{k}_{n}} \; {p}_{n}$.
\end{Lemma}
\begin{proof}
	First note that for any $n$, ${c}^{q}_{{k}_{n} - 1} =
{c}^{{p}_{n}}_{{k}_{n} - 1}$. So since ${p}_{n + 1} \; {\leq}_{{k}_{n}} \;
{p}_{n}$, ${c}^{q}_{{k}_{n} - 1} = {c}^{{p}_{n}}_{{k}_{n} - 1} = {c}^{{p}_{n +
1}}_{{k}_{n} - 1} < {c}^{{p}_{n + 1}}_{{k}_{n}} = {c}^{q}_{{k}_{n}}$. It follows
that for all $m$, ${c}^{q}_{m} < {c}^{q}_{m + 1}$, and so $q$ is a
$0$-condition.

	To check that $q \; {\leq}_{{k}_{n}} \; {p}_{n}$, note that ${s}^{q} =
{s}^{{p}_{n}}$, and that for all $m \in [0, {k}_{n})$, ${c}^{q}_{m} =
{c}^{{p}_{n}}_{m}$. So it is enough to check clauses (8) and (9) of Definition
\ref{def:norm}. For clause (9), simply note that for any $m \in [{k}_{n},
\infty)$, there is a $l > n$ such that ${c}^{q}_{m} = {c}^{{p}_{l}}_{m}$ and
that ${p}_{l} \leq {p}_{n}$. For clause (8) simply note that for any $m \in
\omega$, there is a ${p}_{l} \leq {p}_{n}$ such that ${c}^{q}_{m} =
{c}^{{p}_{l}}_{m}$ and ${c}^{q}_{m + 1} = {c}^{{p}_{l}}_{m + 1}$.  
\end{proof}
Fix $\langle {X}_{\alpha}: \alpha < {\omega}_{1}\rangle$ such that
	\begin{enumerate}
		\item
			${X}_{\alpha}$ is a non-empty closed subset of $\cube$
		\item
			$\A = {\bigcup}_{\alpha < {\omega}_{1}}{{X}_{\alpha}}$
is a MAD family.
	\end{enumerate}
We will be working with forcing extensions of the model in which the codes for
the ${X}_{\alpha}$ live. We adopt the standing convention that when we write
either ``${X}_{\alpha}$'' or ``$\A$'' while working inside such a model we mean
the set that is gotten by interpreting the codes in that model. For each $\alpha
< {\omega}_{1}$, let ${Y}_{\alpha}$ be the closure of ${X}_{\alpha}$ in
$\Pset(\omega)$. Note that ${Y}_{\alpha}$ is compact and that ${Y}_{\alpha}
\setminus {X}_{\alpha} \subset \fin$.
\begin{Def} \label{def:filter}
	Suppose $p$ is a $0$-condition. Define ${A}_{p} = \{s \in \fin: \exists
n \in \omega \[{\nor}_{n}(s \cap {s}_{n}) > 0\]\}$. Let ${\F}_{p}$ be the filter
on $\omega$ generated by the set
		\begin{align*}
			{C}_{p} = \{\omega \setminus a: a \subset \omega \wedge
\neg \exists s \in {A}_{p} \[s \subset a\]\}
		\end{align*}
\end{Def}
All filters on $\omega$ are assumed to contain the Fr{\'e}chet filter. Note that
${C}_{p}$ is a closed subset of $\Pset(\omega)$ and so ${\F}_{p}$ is
${F}_{\sigma}$ in $\Pset(\omega)$. Note also that for any $i \in \omega$, if
$\omega \setminus i \subset {a}_{0} \cup \dotsb \cup {a}_{k}$ and $n \in \omega$
is such that $i \cap {s}_{n} = 0$ and ${\nor}_{n}({s}_{n}) > k + 1$, then for
some $0 \leq l \leq k$ ${\nor}_{n}({a}_{l} \cap {s}_{n}) > 0$, whence $\omega
\setminus {a}_{l} \notin {C}_{p}$. It follows that ${\F}_{p}$ is a proper
filter. Note that for any $s \in {A}_{p}$, $s \cap \interior(p) \neq 0$, and so
$\interior(p) \in {\F}_{p}$. 

Consider the forcing extension of $\V$ obtained by adding ${\omega}_{1}$ Cohen
reals. For each $\delta \leq {\omega}_{1}$, let ${\V}_{\delta}$ denote the
extension by the first $\delta$ many of these. We assume that $\A$ remains MAD
in ${\V}_{{\omega}_{1}}$.
\begin{Lemma} \label{lem:p}
	In ${\V}_{{\omega}_{1}}$, let $\F$ be any ${F}_{\sigma}$ filter and
suppose that $\GG$, the filter generated by $\F \cup \F(\A)$, is a proper
filter. Then $\GG$ is ${P}^{+}$.
\end{Lemma}
\begin{proof}
	Work in ${\V}_{{\omega}_{1}}$. Fix $\langle {b}_{n}: n \in
\omega\rangle$ such that ${b}_{n + 1} \subset {b}_{n}$ and each ${b}_{n} \in
{\GG}^{+}$. Write $\F = {\bigcup}_{n \in \omega}{{\T}_{n}}$, where each
${\T}_{n}$ is a compact subset of $\Pset(\omega)$. Fix $\delta < {\omega}_{1}$
such that $\langle {b}_{n}: n \in \omega\rangle \in {\V}_{\delta}$ and (the code
for) $\langle {\T}_{n}: n \in \omega \rangle \in {\V}_{\delta}$. In
${\V}_{\delta}$, observe that for any ${\alpha}_{0}, \dotsc, {\alpha}_{k} <
{\omega}_{1}$, any $n \in \omega$, any $\left( {a}_{0}, \dotsc {a}_{k}\right)
\in {Y}_{{\alpha}_{0}} \times \dotsb \times {Y}_{{\alpha}_{k}}$, any $c \in
{\T}_{n}$, and any $m \in \omega$, ${b}_{m} \cap c \cap (\omega \setminus
{a}_{0}) \cap \dotsb \cap (\omega \setminus {a}_{k})$ is infinite. Therefore, by
a standard compactness argument, for each ${\alpha}_{0}, \dotsc, {\alpha}_{k} <
{\omega}_{1}$, $n, m, l \in \omega$, there is a finite set $s \subset {b}_{m}
\setminus l$ such that
	\begin{align*}
		\forall \left( {a}_{0}, \dotsc {a}_{k}\right) \in
{Y}_{{\alpha}_{0}} \times \dotsb \times {Y}_{{\alpha}_{k}} \forall c \in
{\T}_{n} \[s \cap c \cap (\omega \setminus {a}_{0}) \cap \dotsb \cap (\omega
\setminus {a}_{k}) \neq 0\]. \tag{$\ast$}  
	\end{align*}   
Note that $(\ast)$ is absolute between ${\V}_{\delta}$ and
${\V}_{{\omega}_{1}}$. Still in ${\V}_{\delta}$, consider the natural poset $\P$
for adding a pseudo-intersection to $\langle {b}_{n}: n \in \omega\rangle$ using
finite conditions. $\P$ is forcing equivalent to Cohen forcing. So in
${\V}_{{\omega}_{1}}$, there is a set $b$ which is $({\V}_{\delta}, \P)$
generic. Clearly, $\forall n \in \omega\[b \; {\subset}^{\ast} \; {b}_{n}\]$.
Also, by genericity, for each ${\alpha}_{0}, \dotsc, {\alpha}_{k} <
{\omega}_{1}$, $n, l \in \omega$, there is $s \subset b \setminus l$ such that
$(\ast)$ holds. Thus $b \in {\GG}^{+}$.   
\end{proof}
\begin{Def} \label{def:utree}
	For an ultrafilter $\U$, a $\U$-tree is a tree $T \subset {\omega}^{<
\omega}$ such that $\forall s \in T \[{\succc}_{T}(s) \in \U\]$ and $\forall f
\in [T] \forall n \in \omega \[f(n) < f(n + 1)\]$. Thus each $f \in [T]$
determines an element of $\cube$ in a natural way. We will often confuse these
below.
\end{Def}
\begin{Lemma} \label{lem:ad}
	In ${\V}_{{\omega}_{1}}$, suppose that $\F$ is a ${F}_{\sigma}$ filter
such that $\GG$, the filter generated by $\F \cup \F(\A)$, is proper. Suppose $b
\in {\GG}^{+}$. Then for each ${\alpha}_{0}, \dotsc, {\alpha}_{k} <
{\omega}_{1}$, there is a $c \in {\[b\]}^{\omega}$ such that $c \in {\GG}^{+}$
and $\forall ({a}_{0}, \dotsc, {a}_{k}) \in {X}_{{\alpha}_{0}} \times \dotsb
\times {X}_{{\alpha}_{k}} \[\lc ({a}_{0} \cup \dotsb \cup {a}_{k}) \cap c \rc <
\omega \]$.
\end{Lemma}
\begin{proof}
	Let $\E$ be the filter generated by $\GG \cup \{b\}$, and let $\I$ be
${\E}^{*}$, the dual ideal. Consider the forcing with $\Pset(\omega) / \I$. By
Lemma \ref{lem:p}, this forcing does not add any reals and adds a P-point $\U
\supset \E$. Work in ${\V}^{\Pset(\omega) / \I}_{{\omega}_{1}}$. Fix $0 \leq i
\leq k$ and let $\I({X}_{{\alpha}_{i}})$ be the ideal generated by
${X}_{{\alpha}_{i}}$. This is analytic. By a theorem of Blass~\cite{Bl}, there
is a $\U$-tree $T$ such that either $[T] \subset \I({X}_{{\alpha}_{i}})$ or $[T]
\cap \I({X}_{{\alpha}_{i}}) = 0$. As $\U$ is a P-point, without loss of
generality, there is a set ${c}_{i} \in {\[b\]}^{\omega} \cap \U$ such that
$\forall s \in T \[ {\succc}_{T}(s) {=}^{*} {c}_{i} \]$. We claim that $\forall
a \in {X}_{{\alpha}_{i}} \[ \lc a \cap {c}_{i} \rc < \omega\]$. Suppose not.
Then it is possible to choose $f \in [T]$ such that $f \in
\I({X}_{{\alpha}_{i}})$. On the other hand, ${c}_{i} \in {\I}^{+}(\A)$. As
$\Pset(\omega) / \I$ adds no new reals, $\A$ is 
MAD in ${\V}^{\Pset(\omega) / \I}_{{\omega}_{1}}$, and so $\existsinf a \in \A
\[\lc a \cap {c}_{i} \rc = \omega\]$. But then it is possible to choose $f \in
[T]$ such that $\existsinf a \in \A \[\lc a \cap f \rc = \omega\]$, whence $f
\notin \I({X}_{{\alpha}_{i}})$. This contradicts the choice of $T$. Now, put $c
= {\bigcap}_{0 \leq i \leq k}{{c}_{i}}$ . $c \in {\[b\]}^{\omega} \cap \U$.
Therefore, $c \in {\GG}^{+}$. Also, it is clear that $\forall ({a}_{0}, \dotsc,
{a}_{k}) \in {X}_{{\alpha}_{0}} \times \dotsb \times {X}_{{\alpha}_{k}} \[\lc
({a}_{0} \cup \dotsb \cup {a}_{k}) \cap c \rc < \omega \]$. Since $\Pset(\omega)
/ \I$ did not add any reals, $c \in {\V}_{{\omega}_{1}}$, and we are done. 
\end{proof}
\begin{Def} \label{def:1condition}
	A $0$-condition $p$ is said to be a \emph{$1$-condition} if for each $a
\in \I(\A)$ and for each $k \in \omega$, there is $n \in \omega$ such that
${\nor}_{n}({s}_{n} \setminus a) \geq k$.
\end{Def}
The next lemma is the major new ingredient in the proof.
Most of the extra work needed to deal with $\ac$ rather than $\a$ is contained
in it.
\begin{Lemma} \label{lem:main}
	Work in ${\V}_{{\omega}_{1}}$. Let $p$ be a $0$-condition and let $c
\subset \omega$. Then the following are equivalent:
		\begin{enumerate}
			\item
				For every ${\alpha}_{0}, \dotsc, {\alpha}_{k} <
{\omega}_{1}$, there exists a $1$-condition $q$ such that $q \; {\leq}_{0} \;
p$, $\forall ({a}_{0}, \dotsc, {a}_{k}) \in {X}_{{\alpha}_{0}} \times \dotsb
\times {X}_{{\alpha}_{k}}\[\lc \interior(q) \cap ({a}_{0} \cup \dotsb \cup
{a}_{k}) \rc < \omega\]$, and $\interior(q) \subset c$.
			\item
				The filter generated by ${\F}_{p} \cup \F(\A)
\cup \{c\}$ is proper.
		\end{enumerate}
\end{Lemma}
\begin{proof}
	Assume (1), and suppose for a contradiction that there exist ${b}_{0},
\dotsc, {b}_{l} \in {C}_{p}$, ${\alpha}_{0}, \dotsc, {\alpha}_{k} <
{\omega}_{1}$, $({a}_{0}, \dotsc, {a}_{k}) \in {X}_{{\alpha}_{0}} \times \dotsb
\times {X}_{{\alpha}_{k}}$, and $i \in \omega$ such that $c \cap \interior(p)
\cap {b}_{0} \cap \dotsb \cap {b}_{l} \cap (\omega \setminus {a}_{0}) \cap
\dotsb \cap (\omega \setminus {a}_{k}) \subset i$. Applying (1), find $q \;
{\leq}_{0} \; p$ such that $\interior(q) \subset c$, and $\interior(q) \cap
({a}_{0} \cup \dotsb \cup {a}_{k})$ is finite. Find $n \in \omega$ such that
${\nor}^{q}_{n}({s}^{q}_{n}) > l + 1$, $i \cap {s}^{q}_{n} = 0$, and $({a}_{0}
\cup \dotsb \cup {a}_{k}) \cap {s}^{q}_{n} = 0$. Since ${s}^{q}_{n} \subset
\interior(p) \cap c$, it follows that ${s}^{q}_{n} \subset (\omega \setminus
{b}_{0}) \cup \dotsb \cup(\omega \setminus {b}_{l})$. But then, for some $0 \leq
j \leq l$, ${\nor}^{q}_{n}((\omega \setminus {b}_{j}) \cap {s}^{q}_{n}) > 0$. So
there must be $m \in \
omega$ such that ${\nor}^{p}_{m}({s}^{p}_{m} \cap (\omega \setminus {b}_{j})
\cap {s}^{q}_{n}) > 0$, whence $(\omega \setminus {b}_{j}) \cap {s}^{q}_{n} \in
{A}_{p}$. This, however, means that ${b}_{j} \notin {C}_{p}$, a contradiction.

	Next, suppose that ${\F}_{p} \cup \F(\A) \cup \{c\}$ generates a proper
filter. We will prove (1). Let $\GG$ denote the filter generated by ${\F}_{p}
\cup \F(\A) \cup \{c\}$. First notice the following things about ${A}_{p}$. If
$s \in {A}_{p}$, then $\lc s \rc > 1$. Next, if $s \subset t$, and $s \in
{A}_{p}$, then $t \in {A}_{p}$. Finally, if $b \in {\GG}^{+}$, then $\exists s
\in {A}_{p} \[s \subset b\]$. Now, we define \emph{the norm induced by
${A}_{p}$}, $\nor: \fin \rightarrow \omega$ by the following clauses:
	\begin{itemize}
		\item
			$\nor(s) \geq 0$, for every $s \in \fin$
		\item
			$\nor(s) \geq 1$ iff $s \in {A}_{p}$
		\item
			for $n > 1$, $\nor(s) \geq n$ iff for every ${s}_{0},
{s}_{1}$ such that $s = {s}_{0} \cup {s}_{1}$, there is $i \in 2$ such that
$\nor({s}_{i}) \geq n - 1$
		\item
			$\nor(s) = \max\{n \in \omega: \nor(s) \geq n\}$.
	\end{itemize}  
It is easy to check that $\nor$ is well defined and is a norm on $\omega$. Next,
we check by induction on $n \in \omega$ that for any $b \in {\GG}^{+}$, $\exists
s \subset b \[\nor(s) \geq n\]$. If $n = 0$, then there is nothing to prove. For
$n = 1$, use the previous observation that $\exists s \in {A}_{p} \[s \subset
b\]$. Suppose that $n > 1$ and that the claim is true for $n - 1$. Suppose for a
contradiction that it fails for $n$. In particular, for every $k \in \omega$,
$\nor(b \cap k) \not\geq n$, and so there exist ${b}^{k}_{0}, {b}^{k}_{1}$ such
that $b \cap k = {b}^{k}_{0} \cup {b}^{k}_{1}$, and neither ${b}^{k}_{0}$ nor
${b}^{k}_{1}$ contains a set $s$ such that $\nor(s) \geq n - 1$. By a standard
K\"onig's Lemma argument, this gives us ${b}_{0}, {b}_{1}$ such that $b =
{b}_{0} \cup {b}_{1}$ and neither ${b}_{0}$ nor ${b}_{1}$ contains a set $s$
with $\nor(s) \geq n - 1$. However, either ${b}_{0}$ or ${b}_{1}$ is in
${\GG}^{+}$, which contradicts the induction hypothesis.

	Now, fix ${\alpha}_{0}, \dotsc, {\alpha}_{k} < {\omega}_{1}$. As
${\F}_{p}$ is a ${F}_{\sigma}$ filter and as $\interior(p) \cap c$ is positive
for the filter generated by ${\F}_{p} \cup \F(\A)$, Lemma \ref{lem:ad} applies
and implies that there is a set $d \in {\[\interior(p) \cap c \]}^{\omega}$
which is positive for the filter generated by ${\F}_{p} \cup \F(\A)$, and
$\forall ({a}_{0}, \dotsc, {a}_{k}) \in {X}_{{\alpha}_{0}} \times \dotsb \times
{X}_{{\alpha}_{k}} \[\lc ({a}_{0} \cup \dotsb \cup {a}_{k}) \cap d \rc < \omega
\]$. Of course, $d \in {\GG}^{+}$. Therefore, for any $a \in \I(\A)$, and for
any $n \in \omega$, there is a $s \subset d$ such that $\nor(s) \geq n$ and $a
\cap s = 0$. Choose $\delta < {\omega}_{1}$ such that $p, c, d$, and $\nor$ are
in ${\V}_{\delta}$. Now, work in ${\V}_{\delta}$. Define a poset $\P$ as
follows. For $s \in {\[d\]}^{< \omega} \setminus \{0\}$, let ${m}_{s}$ denote
$\min\{m \in \omega: s \cap {s}^{p}_{m} \neq 0\}$ and let ${m}^{s}$ denote
$\max\{m \in \omega: s \
cap {s}^{p}_{m} \neq 0\}$. $\P$ consists of all $\sigma : \dom(\sigma)
\rightarrow {\[d\]}^{<\omega}$ such that:
	\begin{itemize}
		\item
			$\dom(\sigma) \in \omega$ and for each $i <
\dom(\sigma)$, $\nor(\sigma(i)) > 0$
		\item
			for any $i < i + 1 < \dom(\sigma)$, $\langle \sigma(i),
\nor \restrict \sigma(i) \rangle < \langle \sigma(i + 1), \nor \restrict
\sigma(i + 1) \rangle$ and also ${m}^{\sigma(i)} < {m}_{\sigma(i + 1)}$.
	\end{itemize}
For $\sigma, \tau \in \P$, $\tau \leq \sigma$ iff $\tau \supset \sigma$. Fix
${\beta}_{0}, \dotsc, {\beta}_{l} < {\omega}_{1}$, $n, m \in \omega$. For any
$({a}_{0}, \dotsc, {a}_{l}) \in {Y}_{{\beta}_{0}} \times  \dotsb \times
{Y}_{{\beta}_{l}}$, there is a $s \subset d \setminus m$ such that $s \cap
({a}_{0} \cup \dotsb \cup {a}_{l}) = 0$ and $\nor(s) \geq n$. Again, by a
compactness argument, there is a set $s \subset d \setminus m$ such that
	\begin{align*}
		\forall ({a}_{0}, \dotsc, {a}_{l}) \in {Y}_{{\beta}_{0}} \times
\dotsb \times {Y}_{{\beta}_{l}} \exists t \subset s \[({a}_{0} \cup \dotsb \cup
{a}_{l}) \cap t = 0 \wedge \nor(t) \geq n\] \tag{$\ast$}.
	\end{align*}
Note that ($\ast$) is absolute between ${\V}_{\delta}$ and
${\V}_{{\omega}_{1}}$. Now, for each ${\beta}_{0}, \dotsc, {\beta}_{l} <
{\omega}_{1}$ and $n \in \omega$, 
	\begin{align*}
		\left\{\tau \in \P: \exists i < \dom(\tau) \[\tau(i) \
\text{satisfies} \ (\ast) \ \text{with respect to} \ {\beta}_{0}, \dotsc,
{\beta}_{l}, n \]\right\}
	\end{align*}
is dense in $\P$. Since $\P$ is forcing equivalent to Cohen forcing, there is a
function $f: \omega \rightarrow {\[d\]}^{< \omega}$ in ${\V}_{{\omega}_{1}}$
which is $({\V}_{\delta}, \P)$-generic. For each $i \in \omega$, put
${c}^{q}_{i} = \langle f(i), \nor \restrict f(i) \rangle$. Put $q = \langle
{s}^{p}, \langle {c}^{q}_{i}: i \in \omega \rangle \rangle$. It is clear that
$q$ is a $0$-condition and that $q \; {\leq}_{0} \; p$. It is also clear that
$\interior(q) \subset c$. By genericity of $f$, for each ${\beta}_{0}, \dotsc,
{\beta}_{l} < {\omega}_{1}$ and $n \in \omega$, there is $i \in \omega$ such
that ${s}^{q}_{i}$ satisfies ($\ast$) with respect to ${\beta}_{0}, \dotsc,
{\beta}_{l}$ and $n$. It follows that $q$ is a $1$-condition, and we are done.
\end{proof}
\begin{Cor} \label{cor:ad}
	There are $1$-conditions. Moreover, given any $1$-condition $p$ and
${\alpha}_{0}, \dotsc, {\alpha}_{k} < {\omega}_{1}$, there is a $1$-condition $q
\leq p$ such that $\forall ({a}_{0}, \dotsc, {a}_{k}) \in {X}_{{\alpha}_{0}}
\times \dotsb \times {X}_{{\alpha}_{k}} \[\lc ({a}_{0} \cup \dotsb \cup {a}_{k})
\cap \interior(q) \rc < \omega\]$. 
\end{Cor}
\begin{proof}
	For the second statement, note that if $p$ is a $1$-condition, then the
filter generated by ${\F}_{p} \cup \F(\A)$ is proper. Now, apply Lemma
\ref{lem:main}. 

	The first statement is a corollary of the proof of Lemma \ref{lem:main}.
For example, let $A = {\[\omega\]}^{\geq2}$, and let $\nor$, the norm induced by
$A$, be defined as in the proof of Lemma \ref{lem:main}. Let $\P$ be defined
(with $d = \omega$) as in the proof of Lemma \ref{lem:main}, leaving out any
mention about ${m}^{\sigma(i)}$ and ${m}_{\sigma(i + 1)}$, which are irrelevant
here. Then an appropriate generic for $\P$ yields a $1$-condition.  
\end{proof}	
From this point on the argument is fairly standard, and follows Shelah
\cite{PIF}.
\begin{Def} \label{def:assorted}
	${\P}_{0} = \{p: p \ \text{is a} \ 0\text{-condition}\}$. ${\P}_{1} =
\{p: p \ \text{is a} \ 1\text{-condition}\}$. The ordering on both ${\P}_{0}$
and ${\P}_{1}$ is $\leq$.

	Fix $p \in {\P}_{0}$. Suppose $t \in {\[\interior(p)\]}^{< \omega}$.
Define ${m}^{p}_{t} = \max\{m \in \omega: {s}^{p}_{m} \cap t \neq 0\}$, with the
convention that ${m}^{p}_{t} = -1$ when $t = 0$. For $t \in
{\[\interior(p)\]}^{< \omega}$ and $n > {m}^{p}_{t}$, $p(t, n)$ is the
$0$-condition defined as follows. ${s}^{p(t, n)} = {s}^{p} \cup t$, and for all
$i \in \omega$, ${c}^{p(t, n)}_{i} = {c}^{p}_{i + n}$. It is clear that $p(t, n)
\leq p$.
\end{Def}
The poset ${\P}_{0}$ is proper and does not add dominating reals. Consult either
\cite{PIF} or \cite{Ab} for a proof of this.
We will work towards showing that ${\P}_{1}$ is proper.
We first make some basic observations about the above definitions. Fix $p \in
{\P}_{0}$ and suppose $q \; {\leq}_{0} \; p$. Suppose $t \in
{\[\interior(q)\]}^{< \omega}$. Then ${m}^{q}_{t} \leq {m}^{p}_{t}$. Moreover,
if $k > {m}^{p}_{t}$, then $q(t, k) \; {\leq}_{0} \; p(t, k)$. Also, suppose
that $p, q \in {\P}_{0}$ with $q \; {\leq}_{0} \; p$. Suppose $t \in
{\[\interior(q)\]}^{< \omega}$ and suppose that $k > {m}^{q}_{t}$ and that $l >
{m}^{p}_{t}$. If for each $m \geq k$, ${s}^{q}_{m} \subset {\bigcup}_{j \in [l,
\infty)}{{s}^{p}_{j}}$, then $q(t, k) \; {\leq}_{0} \; p(t, l)$. To avoid
unnecessary repetitions, all conditions belong to ${\P}_{1}$ from this point on
unless specified. Also, unless specified, we are working inside
${\V}_{{\omega}_{1}}$.
\begin{Lemma} \label{lem:process}
	Let $\mathring{x} \in {\V}^{{\P}_{1}}_{{\omega}_{1}}$ such that
${\forces}_{1}{\mathring{x} \in {\V}_{{\omega}_{1}}}$. Fix $p$, $k \in \omega
\setminus \{0\}$, and $t \subset {\bigcup}_{m \in [0, k)}{{s}^{p}_{m}}$. Then
there is $\bar{p} \; {\leq}_{k} \; p$ such that for any $q \; {\leq}_{k} \;
\bar{p}$, if there exists $r \leq q$ such that ${s}^{r} \setminus {s}^{p} = t$
and ${r}{\forces}_{1}{\mathring{x} = x}$, then ${q(t, k)} {\forces}_{1}
{\mathring{x} = x}$.
\end{Lemma}
\begin{proof}
	$\bar{p}$ is gotten as follows. First suppose that there is a $\bar{q}
\; {\leq}_{0} \; p(t, k)$ and $x \in {\V}_{{\omega}_{1}}$ such that ${\bar{q}}
\; {\forces}_{1} \; {\mathring{x} = x}$. We may assume that
${\nor}^{\bar{q}}_{0}\left({s}^{\bar{q}}_{0} \right) > {\nor}^{p}_{k -
1}\left({s}^{p}_{k - 1}\right)$. Now define $\bar{p}$, by ${s}^{\bar{p}} =
{s}^{p}$, ${c}^{\bar{p}}_{m} = {c}^{p}_{m}$, for $m < k$, and ${c}^{\bar{p}}_{m}
= {c}^{\bar{q}}_{m - k}$, for $m \geq k$. If there is no such $\bar{q}$, then
simply set $\bar{p} = p$. In either case, it is clear that $ \bar{p} \;
{\leq}_{k} \; p$.

	Now, fix $q \; {\leq}_{k} \; \bar{p}$. Note that if the first case
happens above, then $q(t, k) \;  {\leq}_{0} \; \bar{q}$, and so ${q(t, k)} \; 
{\forces}_{1} \; {\mathring{x} = x}$. Suppose $r \leq q$ such that ${s}^{r}
\setminus {s}^{p} = t$ and $y \in {\V}_{{\omega}_{1}}$ such that ${r} \;
{\forces}_{1} \; {\mathring{x} = y}$. First, we claim that the first case must
have happened above. Suppose not. Then $\bar{p} = p$. We may assume that
${s}^{r}_{0} \subset {\bigcup}_{m \in [k, \infty)}{{s}^{q}_{m}}$. But then $r \;
{\leq}_{0} \; p(t,k)$, which contradicts the supposition that the first case did
not occur. So the first case occurs, and therefore, ${q(t, k)} \;  {\forces}_{1}
\; {\mathring{x} = x}$. Again, we may assume that ${s}^{r}_{0} \subset
{\bigcup}_{m \in [k, \infty)}{{s}^{q}_{m}}$. But then $r \; {\leq}_{0} \; q(t,
k)$, whence $x = y$.
\end{proof}
\begin{Lemma} \label{lem:process2}
	Let $\mathring{x} \in {\V}^{{\P}_{1}}_{{\omega}_{1}}$ such that
${\forces}_{1}{\mathring{x} \in {\V}_{{\omega}_{1}}}$. Fix $p$, $k \in \omega
\setminus \{0\}$. There exists $\bar{p} \; {\leq}_{k} \; p$ such that
	\begin{align*}
		& \text{for any} \ q \; {\leq}_{k} \; \bar{p} \ \text{and for any} \ t \subset {\bigcup}_{m \in [0, k)}{{s}^{p}_{m}}, \ \text{if there exists} \ r \; \leq \; q \ \text{and} \tag{${\dagger}_{1}$} \\ & x \in {\V}_{{\omega}_{1}} \ \text{such that} \ {s}^{r} \setminus {s}^{p} = t \ \text{and} \ {r} \; {\forces}_{1} \; {\mathring{x} = x}, \ \text{then} \ {q(t,k)} \; {\forces}_{1} \; {\mathring{x} = x}.  
	\end{align*}
\end{Lemma}
\begin{proof}
	Let ${t}_{0}, \dotsc, {t}_{l}$ enumerate all $t \subset {\bigcup}_{m \in
[0, k)}{{s}^{p}_{m}}$. Now construct a sequence $p = {p}_{-1} \; {}_{k}{\geq} \;
{p}_{0} \; {}_{k}{\geq} \; \dotsb \; {}_{k}{\geq} \; {p}_{l} = \bar{p}$ as
follows. For $-1 \leq i < l$, suppose ${p}_{i} \; {\leq}_{k} \; p$ is given.
Note that ${t}_{i + 1} \subset {\bigcup}_{m \in [0, k)}{{s}^{{p}_{i}}_{m}}$. So
apply Lemma \ref{lem:process} to find ${p}_{i + 1} \; {\leq}_{k} \; {p}_{i}$
such that for any $q \; {\leq}_{k} \; {p}_{i + 1}$, if there are $r \leq q$ and
$x \in {\V}_{{\omega}_{1}}$ such that ${s}^{r} \setminus {s}^{p} = {t}_{i + 1}$
and ${r} \; {\forces}_{1} \; {\mathring{x} = x}$, then ${q({t}_{i + 1}, k)} \;
{\forces}_{1} \; {\mathring{x} = x}$. It is clear that $\bar{p}$ is as needed.
\end{proof}
\begin{Lemma} \label{lem:process3}
	Fix $p \in {\P}_{1}$ and $\mathring{f} \in
{\V}^{{\P}_{1}}_{{\omega}_{1}}$ such that ${\forces}_{1}{\mathring{f} \in
{}^{\omega}{\left({\V}_{{\omega}_{1}}\right)}}$. Then there is a $\bar{p} \;
{\leq}_{0} \; p$ such that
	\begin{align*}
		& \text{for any} \ q \; {\leq}_{0} \; \bar{p}, \ \text{for any}
\ t \in {\[\interior(q)\]}^{< \omega}, \ \text{and for any} \ i \in \omega, \
\text{there is a} \tag{${\dagger}_{2}$} \\ & k > {m}^{q}_{t} \ \text{such that
if there is a} \ r \leq q \ \text{and} \ x \in {\V}_{{\omega}_{1}} \ \text{such
that} \ {s}^{r} \setminus {s}^{p} = t \\ & \text{and} \ {r} \; {\forces}_{1} \;
{\mathring{f}(i) = x}, \ \text{then} \ {q(t, k)} \; {\forces}_{1} \;
{\mathring{f}(i) = x}.  
	\end{align*}
\end{Lemma}
\begin{proof}
Define functions $\Sigma: {\omega}^{< \omega} \rightarrow {\P}_{1}$  and
$\Delta: {\omega}^{< \omega} \setminus \{0\} \rightarrow \omega \setminus \{0\}$
with the following properties:
	\begin{enumerate}
		\item
			$\Sigma(0) = p$ and for each $\sigma \in {\omega}^{< \omega}$ and $j \in \omega$, $\Sigma({\sigma}^{\frown}{\langle j \rangle}) \;{\leq}_{\Delta({\sigma}^{\frown}{\langle j \rangle})} \; \Sigma(\sigma)$
		\item
			for each $\sigma \in {\omega}^{< \omega} \setminus \{0\}$, and for each $j \in \omega$, $\Delta({\sigma}^{\frown}{\langle j \rangle}) > \Delta(\sigma)$. Also, for each $\sigma \in {\omega}^{< \omega}$ and $k \in \omega$, there is a $j \in \omega$ such that $\Delta({\sigma}^{\frown}{\langle j \rangle}) > k$  
		\item
			for each $\sigma \in {\omega}^{< \omega}$, $j \in \omega$, $i < \Delta({\sigma}^{\frown}{\langle j \rangle})$, (${\dagger}_{1}$) holds with $\bar{p}$ as $\Sigma({\sigma}^{\frown}{\langle j \rangle})$, $k$ as $\Delta({\sigma}^{\frown}{\langle j \rangle})$, $p$ as $\Sigma(\sigma)$, and $\mathring{x}$ as $\mathring{f}(i)$. 
	\end{enumerate}
By Lemma \ref{lem:process2} it is possible to define such functions $\Sigma$ and
$\Delta$. Now, fix $g \in \BS$. The hypotheses of Lemma \ref{lem:fusion} are
satisfied when ${p}_{n}$ is taken to be $\Sigma(g \restrict n)$ and ${k}_{n}$ as
$\Delta(g \restrict n + 1)$. Let ${q}_{g}$ be the $0$-condition defined as in
Lemma \ref{lem:fusion}. By Lemma \ref{lem:fusion}, for each $n \in \omega$,
${q}_{g} \; {\leq}_{\Delta(g \restrict n + 1)} \; \Sigma(g \restrict n)$.
Suppose for a moment that there is $g \in \BS$ such that ${q}_{g} \in {\P}_{1}$.
We first check that setting ${q}_{g} = \bar{p}$ does the job. Suppose $q \;
{\leq}_{0} \; {q}_{g}$. Fix $t \in {\[\interior(q)\]}^{< \omega}$ and $i \in
\omega$. Find $n \in \omega$ such that $\Delta(g \restrict n + 1) >
\max\left\{{m}^{{q}_{g}}_{t}, i\right\}$. Observe that ${m}^{q}_{t} < \Delta(g
\restrict n + 1)$. Thus $t \subset {\bigcup}_{m \in [0, \Delta(g \restrict n +
1))}{{s}^{{q}^{g}}_{m}}$. As ${q}_{g} \; {\leq}_{\Delta(g \restrict n + 1)} \;
\Sigma(g \restrict n)$, $t \subset {\bigcup}_{m \in [0, \Delta(g \restrict n +
1))}{{s}^{\Sigma(g \restrict n)}_{m}}$. We know that (${\dagger}_{1}$) holds
with $\bar{p}$ as $\Sigma(g \restrict n + 1)$, $k$ as $\Delta(g \restrict n +
1)$, $p$ as $\Sigma(g \restrict n)$, and $\mathring{x}$ as $\mathring{f}(i)$.
Note that ${q}_{g} \; {\leq}_{\Delta(g \restrict n + 2)} \; \Sigma(g \restrict n
+ 1)$, and so ${q}_{g} \; {\leq}_{\Delta(g \restrict n + 1)} \; \Sigma(g
\restrict n + 1)$. Now, suppose there exists $r \leq q$ and $x \in
{\V}_{{\omega}_{1}}$ such that ${s}^{r} \setminus {s}^{p} = t$ and ${r} \;
{\forces}_{1} \; {\mathring{f}(i) = x}$. Note that ${s}^{p} = {s}^{\Sigma(g
\restrict n)}$ and that $q \leq {q}_{g}$. Therefore, $r \leq {q}_{g}$ and
${s}^{r} \setminus {s}^{\Sigma(g \restrict n)} = t$. Applying (${\dagger}_{1}$),
we conclude that ${q}_{g}(t, \Delta(g \restrict n + 1)) \; {\forces}_{1} \;
{\mathring{f}(i) = x}$. But since $q(t, \Delta(g \restrict n + 1)) \; {\leq}_{0}
\; {q}_{g}(t, \Delta(g \restrict n + 
1))$, $q(t, \Delta(g \restrict n + 1)) \; {\forces}_{1} \; {\mathring{f}(i) =
x}$, and we are done. 

	Therefore, it is enough to find $g \in \BS$ such that ${q}_{g} \in
{\P}_{1}$. Find $\delta < {\omega}_{1}$ such that $\Sigma, \Delta \in
{\V}_{\delta}$. Work in ${\V}_{\delta}$. View ${\omega}^{< \omega}$ as a forcing
poset with $\tau \leq \sigma$ iff $\tau \supset \sigma$. Fix $\sigma \in
{\omega}^{< \omega}$, ${\alpha}_{0}, \dotsc, {\alpha}_{k} < {\omega}_{1}$, and
$n, m \in \omega$. Then for each $({a}_{0}, \dotsc , {a}_{k}) \in
{Y}_{{\alpha}_{0}} \times \dotsb \times {Y}_{{\alpha}_{k}}$, there is $i \in
\omega$ and $t \subset {s}^{\Sigma(\sigma)}_{i}$ with
${\nor}^{\Sigma(\sigma)}_{i}(t) \geq n$ such that $t \cap (m \cup {a}_{0} \cup
\dotsb \cup {a}_{k}) = 0$. Again, by a compactness argument, there is $j \in
\omega$ such that
	\begin{align*}
		\forall ({a}_{0}, \dotsc, {a}_{k}) \in &{Y}_{{\alpha}_{0}} \times \dotsb \times {Y}_{{\alpha}_{k}} \tag{$\ast$} \\ & \exists i \leq j \exists t \subset {s}^{\Sigma(\sigma)}_{i}\[{\nor}^{\Sigma(\sigma)}_{i}(t) \geq n \wedge t \cap (m \cup {a}_{0} \cup \dotsb \cup {a}_{k}) = 0\] 
	\end{align*} 
Note that ($\ast$) is absolute between ${\V}_{\delta}$ and
${\V}_{{\omega}_{1}}$. It follows that for any ${\alpha}_{0}, \dotsc, {\alpha}_{k} < {\omega}_{1}$ and $n, m \in \omega$, the set
\begin{align*}
	\{\tau \in {\omega}^{< \omega} \setminus \{0\}: \Delta(\tau) - 1 \ \text{satisfies} \ (\ast) \ \text{with respect to} \ \tau \restrict \lc \tau \rc - 1, {\alpha}_{0}, \dotsc, {\alpha}_{k}, n, m\}
\end{align*}
is dense in ${\omega}^{< \omega}$. There is a $g \in {\V}_{{\omega}_{1}}$ which
is $({\V}_{\delta}, {\omega}^{< \omega})$-generic. By genericity, for each
${\alpha}_{0}, \dotsc, {\alpha}_{k} < {\omega}_{1}$, and $n, m \in \omega$,
there is a $l \in \omega$ such that $\Delta(g \restrict l + 1) - 1$ satisfies
($\ast$) with respect to $g \restrict l$, ${\alpha}_{0}, \dotsc, {\alpha}_{k},
n, m$. Since ${q}_{g} \; {\leq}_{\Delta(g \restrict l + 1)} \; \Sigma(g
\restrict l)$, it follows that ${q}_{g} \in {\P}_{1}$. 
\end{proof}
An easy corollary of Lemma \ref{lem:process3} is the properness of ${\P}_{1}$.
The details are left to the reader.
\begin{Cor} \label{cor:proper}
	${\P}_{1}$ is proper.
\end{Cor}
We next work towards showing that if ${\P}_{0}$ does not destroy $\A$, then ${\P}_{1}$ does not add dominating reals, and more.
\begin{Def} \label{def:positive}
	Fix $\mathring{f} \in {\V}^{{\P}_{1}}_{{\omega}_{1}}$ such that
${\forces}_{1} \; \mathring{f} \in {}^{\omega}{\left({\V}_{{\omega}_{1}}
\right)}$. Let $p \in {\P}_{1}$ satisfy (${\dagger}_{2}$) of Lemma
\ref{lem:process3} with respect to $\mathring{f}$. For each $i \in \omega$, 
define 
	\begin{align*}
		B(p, \mathring{f}, i) = \left\{t \in {\[\interior(p)\]}^{< \omega}: \exists k > {m}^{p}_{t} \exists x \in {\V}_{{\omega}_{1}}\[p(t,k) \; {\forces}_{1} \; {\mathring{f}(i) = x}\]\right\}. 
	\end{align*}
\end{Def}
Note that if $\mathring{f}$ and $p$ are as in Definition \ref{def:positive}, and
if $q \; {\leq}_{0} \; p$, then $q$ also satisfies (${\dagger}_{2}$) with
respect to $\mathring{f}$ and that $B(q, \mathring{f}, i) =
{\[\interior(q)\]}^{< \omega} \cap B(p, \mathring{f}, i)$, for each $i \in
\omega$.
\begin{Lemma} \label{lem:process4}
	Let $\mathring{f}$ and $p$ be as in Definition \ref{def:positive}. Fix
$k \in \omega \setminus \{0\}$. There exists $\bar{p} \; {\leq}_{k} \; p$ such
that
	\begin{align*}
		\forall t \subset {\bigcup}_{m \in [0, k)}{{s}^{p}_{m}} & \forall i < k \forall m \geq k \tag{${\dagger}_{3}$} \\ & \forall u \subset {s}^{\bar{p}}_{m}  \[{\nor}^{\bar{p}}_{m}(u) > 0 \implies \exists v \subset u \[t \cup v \in B(\bar{p}, \mathring{f}, i)\]\].
	\end{align*}
\end{Lemma}
\begin{proof}
	Let $A$ be the set of all $u \in {\[{\bigcup}_{n \in [k, \infty)}{{s}^{p}_{n}}\]}^{< \omega}$ such that
	\begin{enumerate}
		\item
			for some $m \in \omega$, ${\nor}^{p}_{m}({s}^{p}_{m} \cap u) > 0$
		\item
			for each $t \subset {\bigcup}_{m \in [0, k)}{{s}^{p}_{m}}$ and $i < k$, there exists $v \subset u$ such that $t \cup v \in B(p, \mathring{f}, i)$.
	\end{enumerate}
It is easy to see that for any $u \in A$, $\lc u \rc > 1$ and that if $u \subset
w$, then $w \in A$. Let $\GG$ denote the filter generated by ${\F}_{p} \cup
\F(\A)$. Note that $\GG$ is a proper filter. Fix $c \in {\GG}^{+}$. Then the
filter generated by $\GG \cup \{c\}$ is proper, and so by Lemma \ref{lem:main},
there is a $1$-condition $q \; {\leq}_{0} \; p$ such that $\interior(q) \subset
c$. Let ${n}_{0}$ be least such that for each $n \geq {n}_{0}$, ${s}^{q}_{n}
\subset {\bigcup}_{m \in [k, \infty)}{{s}^{p}_{m}}$, and
${\nor}^{q}_{n}({s}^{q}_{n}) > {\nor}^{p}_{k - 1}({s}^{p}_{k - 1})$. Define
$\bar{q}$ such that ${s}^{\bar{q}} = {s}^{q}$, for each $m \in [0, k)$,
${s}^{\bar{q}}_{m} = {s}^{p}_{m}$, and for all $m \in [k, \infty)$,
${s}^{\bar{q}}_{m} = {s}^{q}_{(m - k) + {n}_{0}}$. It is clear that $\bar{q}$ is
a $1$-condition and that $\bar{q} \;  {\leq}_{k} \; p$. Now, fix $t \subset
{\bigcup}_{m \in [0, k)}{{s}^{p}_{m}}$ and $i < k$. Find $r \leq \bar{q}(t, k)$
and $x \in {\V}_{{\omega}_{1}}$ such that 
$r \; {\forces}_{1} \; \mathring{f}(i) = x$. Let $v = {s}^{r} \setminus \left({s}^{p} \cup t \right)$ and note that since $p$ satisfies (${\dagger}_{2}$), $t \cup v \in B(p, \mathring{f}, i)$. Find $n(t, i) > {n}_{0}$ such that $v \subset {\bigcup}_{m \in [{n}_{0}, n(t, i))}{{s}^{q}_{m}}$. Put $n = \max{\left\{n(t, i): t \subset {\bigcup}_{m \in [0, k)}{{s}^{p}_{m}} \wedge i < k \right\}}$. Let
$u = {\bigcup}_{m \in [{n}_{0}, n)}{{s}^{q}_{m}}$. Observe that $u \in {\[{\bigcup}_{m \in [k, \infty)}{{s}^{p}_{m}}\]}^{< \omega}$. Since
${s}^{q}_{{n}_{0}} \subset u$, (1) is satisfied. Also by the way $n$ is chosen,
(2) is satisfied. Therefore $u \in A$. Since $u \subset \interior(q) \subset c$,
we conclude that for any $c \in {\GG}^{+}$, there is a $u \in A$ such that $u \subset c$.

	Now, let $\nor: \fin \rightarrow \omega$ be the norm induced by $A$,
defined exactly as in the proof of Lemma \ref{lem:main}. Arguing as in Lemma
\ref{lem:main}, it is easy to prove that for any $c \in {\GG}^{+}$ and $n \in \omega$, there is a $s \subset c$ with $\nor(s) \geq n$. Find a $\delta < {\omega}_{1}$ such that $p$ and $\nor$ are in ${\V}_{\delta}$. Working in ${\V}_{\delta}$, define a poset $\P$ as follows. For a non-empty set $u \in {\[\interior(p)\]}^{< \omega}$, ${m}^{u}$ and ${m}_{u}$ are defined as in the
proof of Lemma \ref{lem:main}. A condition in $\P$ is a function $\sigma:
\dom(\sigma) \rightarrow {\[\interior(p)\]}^{< \omega}$ such that
\begin{enumerate}
	\item[(3)]
		$\dom(\sigma) \in \omega$ and for each $i < \dom(\sigma)$, $\sigma(i) \subset {\bigcup}_{m \in [k, \infty)}{{s}^{p}_{m}}$ and $\nor(\sigma(i)) > {\nor}^{p}_{k - 1}({s}^{p}_{k - 1})$.
	\item[(4)]
		for each $i < i + 1 < \dom(\sigma)$, $\langle \sigma(i), \nor \restrict \sigma(i) \rangle < \langle \sigma(i + 1), \nor \restrict \sigma(i + 1) \rangle$, and ${m}^{\sigma(i)} < {m}_{\sigma(i + 1)}$.
\end{enumerate}
For $\sigma, \tau \in \P$, $\tau \leq \sigma$ if $\tau \supset \sigma$. Given
${\alpha}_{0}, \dotsc, {\alpha}_{l} < {\omega}_{1}$, $m, n \in \omega$, and
$({a}_{0}, \dotsc, {a}_{l}) \in {Y}_{{\alpha}_{0}} \times \dotsb \times
{Y}_{{\alpha}_{l}}$, there is a finite $u \subset \interior(p) \setminus m$ such
that $u \cap ({a}_{0} \cup \dotsb \cup {a}_{l}) = 0$ and $\nor(u) \geq n$. So by
a compactness argument, for each ${\alpha}_{0}, \dotsc, {\alpha}_{l} <
{\omega}_{1}$, and $m, n \in \omega$, there is a finite $s \subset \interior(p) \setminus m$ such that
	\begin{align*}
		\forall ({a}_{0}, \dotsc, {a}_{l}) \in {Y}_{{\alpha}_{0}} \times \dotsb \times {Y}_{{\alpha}_{l}} \exists u \subset s\[ u \cap ({a}_{0} \cup \dotsb \cup {a}_{l}) = 0 \wedge \nor(u) \geq n\]. \tag{$\ast$} 
	\end{align*} 
Observe that ($\ast$) is absolute between ${\V}_{\delta}$ and
${\V}_{{\omega}_{1}}$. For each ${\alpha}_{0}, \dotsc, {\alpha}_{l} <
{\omega}_{1}$ and $n \in \omega$, the set
\begin{align*}
	\{\tau \in \P: \exists i < \dom(\tau)\[\tau(i) \ \text{satisfies} \ (\ast) \ \text{with respect to} \ {\alpha}_{0}, \dotsc, {\alpha}_{l}, n\]\}
\end{align*}
is dense in $\P$. In ${\V}_{{\omega}_{1}}$, choose $f: \omega \rightarrow
{\[{\bigcup}_{m \in [k, \infty)}{{s}^{p}_{m}}\]}^{< \omega}$ which is
$({\V}_{\delta}, \P)$-generic. Define $\bar{p}$ as follows. ${s}^{\bar{p}} =
{s}^{p}$. For each $m \in [0, k)$, ${c}^{\bar{p}}_{m} = {c}^{p}_{m}$. For $m \in
[k, \infty)$, ${c}^{\bar{p}}_{m} = \langle f(m - k), \nor \restrict f(m - k) \rangle$. From the genericity of $f$, it follows that $\bar{p}$ is a
$1$-condition. Also, it is clear that $\bar{p} \; {\leq}_{k} \; p$. Now, suppose
that $t \subset {\bigcup}_{m \in [0, k)}{{s}^{p}_{m}}$ and $i < k$. Fix $m \geq
k$ and $u \subset {s}^{\bar{p}}_{m}$ with $\nor(u) > 0$. Then $u \in A$, and so
there is a $v \subset u$ such that $t \cup  v \in B(p, \mathring{f}, i)$. As
$B(\bar{p}, \mathring{f}, i) = {\[\interior(\bar{p})\]}^{< \omega} \cap B(p,
\mathring{f}, i)$, it follows that $t \cup v \in B(\bar{p}, \mathring{f}, i)$.
\end{proof}
Note that if $\bar{p}$ satisfies (${\dagger}_{2}$) with respect to $\mathring{f}$
and it satisfies (${\dagger}_{3}$) with respect to $\mathring{f}$ and $k$, then
any $q \; {\leq}_{k} \; \bar{p}$ also satisfies (${\dagger}_{3}$) with respect
to $\mathring{f}$ and $k$.
\begin{Lemma} \label{lem:process5}
	Let $p$ and $\mathring{f}$ be as in Definition \ref{def:positive}. There
is a $\bar{p} \; {\leq}_{0} \; p$ such that
		\begin{align*}
			& \text{for any} \ i \in \omega, \ \text{there is} \ k > i \ \text{such that for any} \tag{${\dagger}_{4}$} \\ & t \subset {\bigcup}_{m \in [0, k)}{{s}^{\bar{p}}_{m}}, j < k, m \geq k, \ \text{and} \ u \subset {s}^{\bar{p}}_{m}, \ \text{if} \ {\nor}^{\bar{p}}_{m}(u) > 0, \\ &\text{then there exists} \ v \subset u \ \text{such that} \ t \cup v \in B(\bar{p}, \mathring{f}, j).
		\end{align*} 
\end{Lemma}
\begin{proof}
	Define two functions $\Sigma: {\omega}^{< \omega} \rightarrow {\P}_{1}$ and $\Delta: {\omega}^{< \omega} \setminus \{0\} \rightarrow \omega \setminus \{0\}$ with the following properties:
	\begin{enumerate}
		\item
			$\Sigma(0) = p$ and for each $\sigma \in {\omega}^{< \omega}$ and $j \in \omega$, $\Sigma({\sigma}^{\frown}{\langle j \rangle}) \; {\leq}_{\Delta({\sigma}^{\frown}{\langle j \rangle})} \; \Sigma(\sigma)$
		\item
			for each $\sigma \in {\omega}^{< \omega} \setminus \{0\}$, and for each $j \in \omega$, $\Delta({\sigma}^{\frown}{\langle j \rangle}) > \Delta(\sigma)$. Also, for each $\sigma \in {\omega}^{< \omega}$ and $k \in \omega$, there is a $j \in \omega$ such that $\Delta({\sigma}^{\frown}{\langle j \rangle}) > k$
		\item
			for each $\sigma \in {\omega}^{< \omega}$ and $j \in \omega$, $\Sigma({\sigma}^{\frown}{\langle j \rangle})$ satisfies (${\dagger}_{3}$) with respect to $\mathring{f}$ and $\Delta({\sigma}^{\frown}{\langle j \rangle})$.  
	\end{enumerate} 
By Lemma \ref{lem:process4} it is possible to find $\Sigma$ and $\Delta$ with
these properties. Note that for any $\sigma \in {\omega}^{< \omega}$,
$\Sigma(\sigma) \; {\leq}_{0} \; p$. Therefore, $\Sigma(\sigma)$ satisfies
(${\dagger}_{2}$) with respect to $\mathring{f}$. So Lemma \ref{lem:process4}
does apply to each $\Sigma(\sigma)$.

	For each $g \in \BS$, let ${q}_{g}$ be defined exactly as in the proof
of Lemma \ref{lem:process3}. By the same argument as in Lemma
\ref{lem:process3}, there exists $g \in \BS$ such that ${q}_{g} \in {\P}_{1}$.
We argue that putting $\bar{p} = {q}_{g}$ works. Fix $i \in \omega$. Find $n \in
\omega$ such that $\Delta(g \restrict n + 1) > i$. Recall that ${q}_{g} \;
{\leq}_{\Delta(g \restrict n + 1)} \; \Sigma(g \restrict n)$. Moreover, ${q}_{g} \; {\leq}_{\Delta(g \restrict n + 2)} \; \Sigma(g \restrict n + 1)$, and so
${q}_{g} \; {\leq}_{\Delta(g \restrict n + 1)} \; \Sigma(g \restrict n + 1)$. By
(3), $\Sigma(g \restrict n + 1)$ satisfies (${\dagger}_{3}$) with respect to
$\mathring{f}$ and $\Delta(g \restrict n + 1)$. Also, $\Sigma(g \restrict n +
1)$ satisfies (${\dagger}_{2}$) with respect to $\mathring{f}$. It follows that
${q}_{g}$ satisfies (${\dagger}_{3}$) with respect to $\mathring{f}$ and
$\Delta(g \restrict n + 1)$, and we are done.     
\end{proof}
\begin{Lemma} \label{lem:extensions}
Assume that ${\forces}_{0}{\A \ \text{is MAD}}$. Let $p \in {\P}_{1}$. There
exists $\{{a}_{n}: n \in \omega\} \subset \A$ and $\{{q}_{n}: n \in \omega\}
\subset {\P}_{0}$ such that
	\begin{enumerate}
		\item
			$\forall n < {n}^{\ast} \[{a}_{n} \neq
{a}_{{n}^{\ast}}\]$
		\item
			$\forall n \in \omega \[{q}_{n} \; {\leq}_{0} \; p \wedge \interior({q}_{n}) \subset {a}_{n}\]$
	\end{enumerate}
\end{Lemma}
\begin{proof}
	Let $\mathring{x}$ be the canonical ${\P}_{0}$-name for the generic
subset of $\omega$ added by ${\P}_{0}$. Fix $n \in \omega$ and suppose that
$\{{a}_{i}: i < n\} \subset \A$ and $\{{q}_{i}: i < n\} \subset {\P}_{0}$ are
given. We will show how to get ${a}_{n}$ and ${q}_{n}$. Put $a = {\bigcup}_{i <
n}{{a}_{i}}$. Then $a \in \I(\A)$. Put $c = \interior(p) \setminus a$. As $p$ is
a $1$-condition, the filter generated by ${\F}_{p} \cup \F(\A) \cup \{c\}$ is
proper. Apply Lemma \ref{lem:main} to find a $1$-condition $\bar{p} \;
{\leq}_{0} \; p$ with $\interior(\bar{p}) \subset c$. Since ${\forces}_{0}{\A \ \text{is MAD}}$, there is a $0$-condition $q \leq \bar{p}$ and $\alpha < {\omega}_{1}$ such that ${q} \; {\forces}_{0} \; \exists {a}^{\ast} \in {X}_{\alpha} \[\lc {a}^{\ast} \cap \mathring{x} \rc = \omega \]$. Note that for any $r \in {\P}_{0}$, $ {r} \; {\forces}_{0} \; {\mathring{x} \; {\subset}^{\ast} \; \interior(r)}$. It follows that there can be no $r \in {\P}_{0}$ with $r \; {\leq}_{0} \; q$ such that $\forall {a}^{\ast} \in {X}_{\alpha} \[\lc \interior(r) \cap {a}^{\ast} \rc < \omega\]$. By Lemma \ref{lem:main}, this means that the filter generated by ${\F}_{q} \cup \F(\A)$ is not proper. Fix ${b}_{0}, \dotsc, {b}_{l} \in {C}_{q}$, ${a}^{\ast}_{0}, \dotsc {a}^{\ast}_{k} \in \A$, and $i \in \omega$ such that ${b}_{0} \cap \dotsb \cap {b}_{l} \cap (\omega \setminus {a}^{\ast}_{0}) \cap \dotsb \cap (\omega \setminus {a}^{\ast}_{k}) \cap \interior(q) \subset i$. Fix ${m}_{0} \in \omega$ such that for all $m \geq
{m}_{0}$, ${s}^{q}_{m} \cap i = 0$, and ${\nor}^{q}_{m}({s}^{q}_{m}) > \max\{l, k\} + 1$. As ${b}_{j} \in {C}_{q}$ for any $0 \leq j \leq l$, it follows that for any $m \geq {m}_{0}$ there is a ${j}_{m}$ with $0 \leq {j}_{m} \leq k$ such that ${\nor}^{q}_{m}\left({a}^{\ast}_{{j}_{m}} \cap {s}^{q}_{m} \right) \geq (m - {m}_{0}) + 1$. So there is an infinite $X \subset [{m}_{0}, \infty)$ and $0
\leq j \leq k$ such that for each $m \in X$, ${j}_{m} = j$. Put ${a}_{n} = {a}^{\ast}_{j}$. Note that ${a}_{n} \cap \interior(\bar{p}) \neq 0$, and so ${a}_{n} \neq {a}_{i}$ for any $i< n$. Define ${q}_{n}$ as follows. ${s}^{{q}_{n}} = {s}^{\bar{p}} = {s}^{p}$.
Choose ${l}_{0} < {l}_{1} < \dotsb$, with ${l}_{i} \in X$ such that
${\nor}^{q}_{{l}_{i}}\left({a}_{n} \cap {s}^{q}_{{l}_{i}}\right) <
{\nor}^{q}_{{l}_{i + 1}}\left({a}_{n} \cap {s}^{q}_{{l}_{i + 1}}\right)$. For
each $i \in \omega$, define ${c}^{{q}_{n}}_{i} = \langle {a}_{n} \cap
{s}^{q}_{{l}_{i}}, {\nor}^{q}_{{l}_{i}} \restrict \left({a}_{n} \cap
{s}^{q}_{{l}_{i}} \right) \rangle$. As $q \leq \bar{p}$, it is clear that
${q}_{n} \; {\leq}_{0} \; \bar{p} \; {\leq}_{0} \; p$. Also, $\interior({q}_{n})
\subset {a}_{n}$, and so ${q}_{n}$ and ${a}_{n}$ are as needed.     
\end{proof}
\begin{Lemma} \label{lem:niceform}
	Assume that ${\forces}_{0}{\A \ \text{is MAD}}$. Let $\mathring{f}$ be
as in Definition \ref{def:positive}. Suppose that $p \in {\P}_{1}$ satisfies
both $({\dagger}_{2})$ and $({\dagger}_{4})$ with respect to $\mathring{f}$.
There exists a $1$-condition $q \; {\leq}_{0} \; p$ and $\{{a}_{n}: n \in \omega\} \subset \A$ with the following properties:
	\begin{enumerate}
		\item
			for all $n < {n}^{\ast}$, ${a}_{n} \neq {a}_{{n}^{\ast}}$
		\item
			for each $n, l \in \omega$, $\forallbutfin m \in \omega \exists t \subset {s}^{q}_{m} \[{\nor}^{q}_{m}(t) \geq l \wedge t \subset {a}_{n}\]$
		\item
			for any $k \in \omega$, $t \subset {\bigcup}_{m \in [0, k)}{{s}^{q}_{m}}$, and $u \subset {s}^{q}_{k}$, if ${\nor}^{q}_{k}(u) > 0$, then there exists $v \subset u$ and $x \in {\V}_{{\omega}_{1}}$ such that $q(t \cup v, k + 1) \; {\forces}_{1} \; \mathring{f} (k) = x$.
	\end{enumerate}    
\end{Lemma}
\begin{proof}
First apply Lemma \ref{lem:extensions} to $p$ to find $\{{a}_{n}: n \in
\omega\}$ and $\{{q}_{n}: n \in \omega\} \subset {\P}_{0}$ satisfying (1) and
(2) of Lemma \ref{lem:extensions}. Define $A = \{s \in \fin: \exists n \in
\omega \exists m \in \omega \[{\nor}^{{q}_{n}}_{m}(s \cap {s}^{{q}_{n}}_{m}) >
0\]\}$. Note that for any $s \in A$, $\lc s \rc > 1$ and that if $s \subset t$,
then $t \in A$. Moreover, for any $s \in A$, there is $m \in \omega$ such that
${\nor}^{p}_{m}(s \cap {s}^{p}_{m}) > 0$. Let $\nor: \fin \rightarrow \omega$ be
the norm on $\omega$ induced by $A$, defined as in the proof of Lemma
\ref{lem:main}. Note that for any $n, m \in \omega$ and $s \subset
{s}^{{q}_{n}}_{m}$, $\nor(s) \geq {\nor}^{{q}_{n}}_{m}(s)$. Next, recalling that
$p$ satisfies (${\dagger}_{2}$) with respect to $\mathring{f}$, for each $i \in
\omega$ and $t \in B(p, \mathring{f}, i)$, fix ${k}^{i}_{t} > {m}^{p}_{t}$ such
that $\exists x \in {\V}_{{\omega}_{1}}\[{p\left(t, {k}^{i}_{t} \right)} \;
{\forces}_{1} \; {\mathring{f}(i) = x} \]$.

	Now, to get $q$ proceed as follows. ${s}^{q} = {s}^{p}$. For each $i \in
\omega$ choose ${m}_{0}, \dotsc, {m}_{i} \in \omega$ such that putting
${s}^{q}_{i} = {s}^{{q}_{0}}_{{m}_{0}} \cup \dotsb \cup {s}^{{q}_{i}}_{{m}_{i}}$, the following properties hold:
	\begin{enumerate}
		\item[(4)]
			for each $i < i + 1$, $\max({s}^{q}_{i}) < \min({s}^{q}_{i + 1})$, $\max\{n \in \omega: {s}^{p}_{n} \cap {s}^{q}_{i} \neq 0\} < \min\{n \in \omega: {s}^{p}_{n} \cap {s}^{q}_{i + 1} \neq 0\}$, and $\nor({s}^{q}_{i}) < \nor({s}^{q}_{i + 1})$. (Recall that for all $j \in \omega$, ${q}_{j} \; {\leq}_{0} \; p$; therefore, for any fixed $i \in \omega$, ${s}^{{q}_{j}}_{{m}_{j}} \subset \interior(p)$, for each $0 \leq j \leq i$; also, ${s}^{{q}_{0}}_{{m}_{0}} \subset {s}^{q}_{i}$; therefore, ${s}^{q}_{i}$ is a non-empty finite subset of $\interior(p)$).
		\item[(5)]
			for each $i \in \omega$, for each $t \subset {\bigcup}_{m \in [0, i)}{{s}^{q}_{m}}$, for each $u \subset {s}^{q}_{i}$, if $\nor(u) > 0$, then there exists $v \subset u$ such that $t \cup v \in B(p, \mathring{f}, i)$.  
		\item[(6)]
			for each $i \in \omega$, each $t \subset {\bigcup}_{m \in [0, i)}{{s}^{q}_{m}}$, and each $v \subset {s}^{q}_{i}$ such that $t \cup v \in B(p, \mathring{f}, i)$, $\forall m \geq i + 1 \[{s}^{q}_{m} \subset {\bigcup}_{n \in \[{k}^{i}_{(t \cup v)}, \infty \right)}{{s}^{p}_{n}}\]$.
		\item[(7)]
			for each $i \in \omega$ and $0 \leq j \leq i$, ${\nor}^{{q}_{j}}_{{m}_{j}}\left({s}^{{q}_{j}}_{{m}_{j}}\right) \geq i$.	
	\end{enumerate} 
Before showing how to do this for each $i \in \omega$, let us argue that it is
enough to do so. First note that for any $j \in \omega$, ${q}_{j} \; {\leq}_{0}
\; p$, and so ${s}^{{q}_{j}} = {s}^{p}$. So since for any $i \in \omega$ and $l
\in {s}^{q}_{i}$, there is some $0 \leq j \leq i$ such that $l \in
{s}^{{q}_{j}}_{{m}_{j}}$, it follows that for all ${l}^{\ast} \in {s}^{q}$,
${l}^{\ast} < l$. Next, for any $i \in \omega$, ${s}^{{q}_{0}}_{{m}_{0}} \subset
{s}^{q}_{i}$. So $0 < {\nor}^{{q}_{0}}_{{m}_{0}}\left({s}^{{q}_{0}}_{{m}_{0}}\right) \leq
\nor\left({s}^{{q}_{0}}_{{m}_{0}}\right) \leq \nor({s}^{q}_{i})$. Therefore, if
we put ${c}^{q}_{i} = \langle {s}^{q}_{i}, \nor \restrict {s}^{q}_{i} \rangle$,
then $q = \langle {s}^{q}, \langle {c}^{q}_{i}: i \in \omega \rangle \rangle$ is
a $0$-condition, and $q \; {\leq}_{0} \; p$. To check (2), fix $n, l \in
\omega$. Suppose that $m \geq \max \{n, l\}$. Then there exists ${m}_{n} \in
\omega$ such that ${s}^{{q}_{n}}_{{m}_{n}} \subset {s}^{q}_{m}$, and ${\nor}^{{
q}_{n}}_{{m}_{n}} \left( {s}^{{q}_{n}}_{{m}_{n}}\right) \geq m \geq l$. However,
${s}^{{q}_{n}}_{{m}_{n}} \subset {a}_{n}$ and
$\nor\left({s}^{{q}_{n}}_{{m}_{n}}\right) \geq {\nor}^{{q}_{n}}_{{m}_{n}}\left({s}^{{q}_{n}}_{{m}_{n}}\right) \geq l$. This
verifies (2). 

	Using (2), it is easy to check that $q$ is a $1$-condition. With the
next lemma in mind, we will verify a slightly stronger statement. Fix $X \in
\cube$. Define ${q}_{X} = \langle {s}^{q}, \langle {c}^{q}_{i}: i \in X \rangle
\rangle$. It is clear that ${q}_{X}$ is a $0$-condition and that ${q}_{X} \leq
q$. We check that it is a $1$-condition. Fix $a \in \I(\A)$ and $l \in \omega$.
Fix $n, k \in \omega$ such that $a \cap {a}_{n} \subset k$. Choose $m \in X$
such that ${s}^{q}_{m} \cap k = 0$ and there exists $t \subset {s}^{q}_{m}$ such
that $t \subset {a}_{n}$ and ${\nor}^{q}_{m}(t) \geq l$. It is clear that $t
\cap a = 0$, and this checks that ${q}_{X}$ is a $1$-condition.   

	For (3), fix $i \in \omega$, $t \subset {\bigcup}_{m \in [0,
i)}{{s}^{q}_{m}}$, and $u \subset {s}^{q}_{i}$ such that ${\nor}^{q}_{i}(u) >
0$. By (5), there is a $v \subset u$ such that $t \cup v \in B(p, \mathring{f},
i)$. By (6), for each $m \geq i + 1$, ${s}^{q}_{m} \subset {\bigcup}_{n \in
\[{k}^{i}_{(t \cup v)}, \infty \right)}{{s}^{p}_{n}}$. Note that ${m}^{q}_{(t
\cup v)} \leq i < i + 1$ and ${k}^{i}_{\left( t \cup v \right)} >
{m}^{p}_{\left( t \cup v\right)}$ by definition. Since $q \; {\leq}_{0} \; p $,
it follows that ${q(t \cup v, i + 1)} \; {\leq}_{0} \; {p\left(t \cup v,
{k}^{i}_{\left(t \cup v \right)}\right)}$. Since there exists $x \in
{\V}_{{\omega}_{1}}$ such that ${p\left(t \cup v, {k}^{i}_{\left( t \cup v
\right)}\right)} \; {\forces}_{1} \; {\mathring{f}(i) = x}$, this verifies (3). 

	Finally, we show how to get such ${m}_{0}, \dotsc, {m}_{i} \in \omega$
for each $i \in \omega$. Fix $i \in \omega$, and assume that ${s}^{q}_{j}$ for
$j < i$ are given to us. First, fix ${k}_{0} \in \omega$ such that for each $0
\leq n \leq i$ and for each $k \geq {k}_{0}$,
${\nor}^{{q}_{n}}_{k}\left({s}^{{q}_{n}}_{k}\right) \geq i$ and $\forall j < i
\[\nor({s}^{q}_{j}) < {\nor}^{{q}_{n}}_{k}\left({s}^{{q}_{n}}_{k}\right)\]$.
Also fix ${l}_{0} \geq i$ such that for all $j < i$, ${s}^{q}_{j} \subset
{\bigcup}_{m \in [0, {l}_{0})}{{s}^{p}_{m}}$. Recall that $p$ satisfies
(${\dagger}_{4}$) with respect to $\mathring{f}$. Applying (${\dagger}_{4}$) to
${l}_{0}$, find ${k}_{1} > {l}_{0}$ as in (${\dagger}_{4}$). Next, choose
${k}_{2} \geq {k}_{1}$ such that for each $j < i$, $t \subset {\bigcup}_{m \in
[0, j)}{{s}^{q}_{m}}$, and $v \subset {s}^{q}_{j}$ such that $t \cup v \in B(p,
\mathring{f}, j)$, ${k}_{2} \geq {k}^{j}_{\left( t \cup v\right)}$. Finally,
recall that for each $0 \leq n \leq i$, ${q}_{n} \; {\leq}_{0}\; p$. So it is possible to choose ${k}_{3} \geq {k}_{0}$ such that for each $0 \leq n \leq i$ and each $k \geq {k}_{3}$, ${s}^{{q}_{n}}_{k} \subset {\bigcup}_{n \in \[{k}_{2}, \infty\right)}{{s}^{p}_{n}}$. Now choose ${m}_{0}, \dotsc, {m}_{i} \geq {k}_{3}$. It is easy to see that (4), (6), and (7) are satisfied. For (5), fix $t \subset {\bigcup}_{m \in [0, i)}{{s}^{q}_{m}}$, and $u \subset {s}^{q}_{i}$ with $\nor(u) > 0$. Note that $t \subset {\bigcup}_{m \in [0, {k}_{1})}{{s}^{p}_{m}}$ and that $i \leq {l}_{0} < {k}_{1}$. Moreover, ${s}^{q}_{i} \subset {\bigcup}_{n \in \[{k}_{2} , \infty \right)}{{s}^{p}_{n}}$.
So there exists $m \geq {k}_{2}$ such that ${\nor}^{p}_{m}(u \cap {s}^{p}_{m}) >
0$. We have that $m \geq {k}_{2} \geq {k}_{1}$, $u \cap {s}^{p}_{m} \subset
{s}^{p}_{m}$ and ${\nor}^{p}_{m}(u \cap {s}^{p}_{m}) > 0$. Therefore, by
(${\dagger}_{4}$), there is $v \subset u \cap {s}^{p}_{m} \subset u$ such that
$t \cup v \in B(p, \mathring{f}, i)$, and we are done.     
\end{proof}
\begin{Def} \label{def:almost}
	A poset $\P$ is said to be \emph{almost $\BS$-bounding} if for any $p
\in \P$ and $\mathring{f} \in {\V}^{\P}$ such that $\forces \mathring{f} \in
\BS$, there exist $q \leq p$ and $g \in \BS$ such that for any $X \in \cube$,
there exists ${q}_{X} \leq q$ such that ${{q}_{X}} \; \forces \;{ \existsinf n
\in X \[\mathring{f}(n) \leq g(n)\]}$.
\end{Def}
It is not difficult to see that an almost $\BS$-bounding poset preserves all
$\sigma$-directed unbounded families of monotonic functions in $\BS$. Shelah
proved that a countable support iteration of proper almost $\BS$-bounding posets
does not add a dominating real. 
He also proved that ${\P}_{0}$ is almost $\BS$-bounding (consult either
\cite{PIF} or \cite{Ab}).
\begin{Lemma}\label{lem:almost}
Assume that ${\forces}_{0}{\A \ \text{is MAD}}$. Then ${\P}_{1}$ is almost
$\BS$-bounding.
\end{Lemma}
\begin{proof}
Fix $\mathring{f} \in {\V}^{{\P}_{1}}_{{\omega}_{1}}$ such that ${\forces}_{1}\; {\mathring{f} \in \BS}$ and $p \in {\P}_{1}$. Find $q \; {\leq}_{0} \; p$ as
in Lemma \ref{lem:niceform}. Define $g \in \BS$ as follows. For any $k \in
\omega$ define $g(k) = \max({X}_{k})$, where
	\begin{align*}
		{X}_{k} = \left\{l \in \omega: \exists t \subset {\bigcup}_{m \in [0, k)}{{s}^{q}_{m}} \exists v \subset {s}^{q}_{k} \[q(t \cup v, k + 1) \; {\forces}_{1} \; \mathring{f}(k) = l \]\right\}.
	\end{align*} 
Note that ${X}_{k}$ is non-empty and finite, so $g(k)$ is well-defined. Now, fix
$X \in \cube$ and let ${q}_{X}$ be defined as in the proof of Lemma
\ref{lem:niceform}. Then ${q}_{X} \in {\P}_{1}$ and ${q}_{X} \leq q$. Fix $r
\leq {q}_{X}$ and $n \in \omega$. Fix ${k}^{\ast} \geq n$ such that $t = {s}^{r}
\setminus {s}^{q} \subset {\bigcup}_{m \in [0, {k}^{\ast})}{{s}^{q}_{m}}$.
Choose $i \in \omega$ such that ${s}^{r}_{i} \subset {\bigcup}_{m \in
[{k}^{\ast}, \infty)}{{s}^{q}_{m}}$. There must be $k \in X$ with $k \geq
{k}^{\ast}$ such that ${\nor}^{q}_{k}({s}^{q}_{k} \cap {s}^{r}_{i}) > 0$. It
follows that there exists $l \in {X}_{k}$  and $v \subset {s}^{q}_{k} \cap
{s}^{r}_{i}$ such that $q(t \cup v, k + 1) \; {\forces}_{1} \; \mathring{f}(k) =
l$. But it is clear that $r(v, i + 1) \leq q(t \cup v, k + 1)$. So $r(v, i + 1)
\leq r$ and $r(v, i + 1) \; {\forces}_{1} \; \mathring{f}(k) = l \leq g(k)$.
Since $k \in X$ and $k \geq n$, we are done.
\end{proof}
We now have all the lemmas needed to give a proof of
\begin{Theorem} \label{thm:main1}
It is consistent to have ${\aleph}_{1} = \b < \ac = {\aleph}_{2}$.
\end{Theorem}
\begin{proof}
	Start with a ground model satisfying $\CH$.
Fixing a book-keeping device to ensure that all names for witnesses to $\ac =
{\aleph}_{1}$ are eventually taken care of, do a CS iteration $\langle
{\P}_{\alpha}, {\mathring{\Q}}_{\alpha}: \alpha \leq {\omega}_{2} \rangle$ of
proper almost $\BS$-bounding posets as follows.
At a stage $\alpha < {\omega}_{2}$ suppose that ${\P}_{\alpha}$ is given.
Let ${G}_{\alpha}$ be $(\V, {\P}_{\alpha})$-generic.
In $\V\[{G}_{\alpha}\]$, let $\langle {X}^{\alpha}_{\xi}: \xi <
{\omega}_{1}\rangle$ be a sequence of non-empty closed subsets of $\cube$ given
by the book-keeping device.
If $\A = {\bigcup}_{\xi < {\omega}_{1}}{{X}^{\alpha}_{\xi}}$ is not MAD, then let
${\Q}_{\alpha}$ be the trivial poset. 
Now assume that $\A$ is MAD.
Let ${\CC}_{{\omega}_{1}}$ be the poset for adding ${\omega}_{1}$ Cohen reals.
Let $H$ be $(\V\[{G}_{\alpha}\], {\CC}_{{\omega}_{1}})$-generic.
If $\A$ is not MAD in $\V\[{G}_{\alpha}\]\[H\]$, then in $\V\[{G}_{\alpha}\]$,
let ${\Q}_{\alpha} = {\CC}_{{\omega}_{1}}$.
Suppose $\A$ is MAD in $\V\[{G}_{\alpha}\]\[H\]$.
In $\V\[{G}_{\alpha}\]\[H\]$, if there exists $p \in {\P}_{0}$ such that $p \;
{\forces}_{0} \; \A \ \text{is not MAD}$, then let $\R = \{q \in {\P}_{0}: q \leq p\}$.
If ${\forces}_{0} \A \ \text{is MAD}$, then let $\R = {\P}_{1}$ (defined with respect to $\A$).
In either case, in $\V\[{G}_{\alpha}\]$ let $\mathring{\R}$ be a full
${\CC}_{{\omega}_{1}}$ name for $\R$.
Let ${\Q}_{\alpha} = {\CC}_{{\omega}_{1}} \ast \mathring{\R}$.
Note that in all of these cases ${\forces}_{{\Q}_{\alpha}}{\A \ \text{is not
MAD}}$.
In $\V$, let ${\mathring{\Q}}_{\alpha}$ be a full ${\P}_{\alpha}$ name for
${\Q}_{\alpha}$.
This completes the definition of the iteration.
If ${G}_{{\omega}_{2}}$ is $(\V, {\P}_{{\omega}_{2}})$ generic, then since
${\P}_{{\omega}_{2}}$ does not add a dominating real, $\b = {\omega}_{1}$ in
$\V\[{G}_{{\omega}_{2}}\]$.
Suppose for a contradiction that $\langle {X}_{\xi}: \xi < {\omega}_{1} \rangle$
is a sequence of non-empty closed subsets of $\cube$ such that $\A =
{\bigcup}_{\xi < {\omega}_{1}}{{X}_{\xi}}$ is MAD.
For some $\alpha < {\omega}_{2}$, the book-keeping device ensured that $\langle
{X}_{\xi}: \xi < {\omega}_{1} \rangle$ was considered at stage $\alpha$.
So there is a set $c \in \cube \cap \V\[{G}_{\alpha + 1}\]$ such that in
$\V\[{G}_{\alpha + 1}\]$, for each $\xi < {\omega}_{1}$, $c$ is almost disjoint
from every element of ${X}_{\xi}$. 
For any fixed $\xi < {\omega}_{1}$, this statement is ${\Pi}^{1}_{1}$ and hence
absolute.
So in $\V\[{G}_{{\omega}_{2}}\]$, for any $\xi < {\omega}_{1}$, $c$ is almost
disjoint from every element of ${X}_{\xi}$.
This is a contradiction. 
\end{proof}

\section{A characterization of ${\P}_{0}$} \label{sec:twostep}
	In this section we show that the poset ${\P}_{0}$ defined in Section
\ref{sec:blessthanaclosed}, which was used by Shelah in \cite{b<s} to produce
the first consistency proof of $\b < \s$, can be viewed as a two step iteration
of a countably closed forcing followed by a $\sigma$-centered poset. 
\begin{Def} \label{def:f}
	Let $\FF = \{\F: \F \ \text{is a proper} \ {F}_{\sigma} \ \text{filter
on} \ \omega\}$. Recall our convention that all filters are required to contain
the Fr{\'e}chet filter. We order $\FF$ by $\supset$. It is clear that $\FF$ is
countably closed and adds an ultrafilter on $\omega$. Let $\mathring{\U}$ denote
the canonical $\FF$-name for the ultrafilter added by $\FF$. For any filter
$\U$, let $\Mi(\U)$ denote the Mathias-Prikry forcing with $\U$. 
\end{Def} 
In this section we will prove that ${\P}_{0}$ is forcing equivalent to $\FF \ast
\Mi(\mathring{\U})$. This is entirely analogous to the characterization of
Mathias forcing as first adding a selective ultrafilter with $\Pset(\omega) /
\FIN$ and then doing Mathias-Prikry forcing with that selective ultrafilter.
Note that $\Pset(\omega) / \FIN$ is forcing equivalent to the partial order of
all countably generated filters on $\omega$ ordered by $\supset$. So $\FF$ is a
natural generalization of $\Pset(\omega) / \FIN$. Our first lemma is rather
well-known.
\begin{Lemma} \label{lem:closed}
	Let $\F$ be a proper ${F}_{\sigma}$ filter on $\omega$. There is a
non-empty closed set $C \subset \Pset(\omega)$ such that $C \subset \F$ and
$\forall b \in \F \exists c \in C \[c \; {\subset}^{\ast} \; b\]$.
\end{Lemma}
\begin{proof}
	Write $\F = {\bigcup}_{n \in \omega}{{\T}_{n}}$, where each ${\T}_{n}$
is a closed subset of $\Pset(\omega)$. Let $C = \{b \cup n: n \in \omega \wedge
b \in {\T}_{n}\}$. It is clear that $\forall b \in \F \exists c \in C \[c \;
{\subset}^{\ast} \; b\]$ and that $C \subset \F$. Note also that $\omega \in C$.
We will check that $C$ is closed. Suppose $\langle {c}_{i}: i \in \omega
\rangle$ is a sequence of elements of $C$ converging to some $c \in
\Pset(\omega)$. For each $i \in \omega$ fix ${n}_{i} \in \omega$ and ${b}_{i}
\in {\T}_{{n}_{i}}$ such that ${c}_{i} = {b}_{i} \cup {n}_{i}$. By passing to a
subsequence, we may assume that the ${b}_{i}$ converge to some $b \in
\Pset(\omega)$ and that either $\forall i \in \omega \[{n}_{i} < {n}_{i + 1}\]$
or there is a fixed $n \in \omega$ such that $\forall i \in \omega \[{n}_{i} =
n\]$. In the first case $c = \omega$, and so $c \in C$. In the second case, each
${b}_{i} \in {\T}_{n}$ and so $b \in {\T}_{n}$. $c = b \cup n$, whence $c \in
C$.
\end{proof}
\begin{Theorem} \label{thm:p0isfstarmu}
There is a dense embedding of ${\P}_{0}$ into $\FF \ast \Mi(\mathring{\U})$.
\end{Theorem}
\begin{proof}
	Most of the tools needed to prove this have already been developed in
the proof of Lemma \ref{lem:main}. Fix $p \in {\P}_{0}$. Let ${A}_{p}$,
${C}_{p}$, and ${\F}_{p}$ be as in Definition \ref{def:filter}. As observed in
Section \ref{sec:blessthanaclosed}, $\interior(p) \in {\F}_{p}$. It follows that
${\F}_{p} \; {\forces}_{\FF} \; \interior(p) \in \mathring{\U}$, and so $\langle
{\F}_{p}, \langle {s}^{p}, \interior(p) \rangle\rangle$ is a condition in $\FF
\ast \Mi(\mathring{\U})$. Define a map $\phi: {\P}_{0} \rightarrow \FF \ast
\Mi(\mathring{\U})$ by $\phi(p) = \langle {\F}_{p}, \langle {s}^{p},
\interior(p) \rangle\rangle$. We will check that $\phi$ is a dense embedding.

	First suppose that $q \leq p$. We must show that $\phi(q) \leq \phi(p)$.
Note that ${s}^{q} \supset {s}^{p}$, $\interior(q) \subset \interior(p)$, and
that ${s}^{q} \setminus {s}^{p} \subset \interior(p)$. So it suffices to show
that ${\F}_{q} \supset {\F}_{p}$. For this, suppose that $s \in {A}_{q}$. Then
there is $n \in \omega$ such that ${\nor}^{q}_{n}(s \cap {s}^{q}_{n}) > 0$. As
$q \leq p$, there must be $m \in \omega$ such that ${\nor}^{p}_{m}(s \cap
{s}^{q}_{n} \cap {s}^{p}_{m}) > 0$. Therefore, ${\nor}^{p}_{m}(s \cap
{s}^{p}_{m}) > 0$, and so $s \in {A}_{p}$. So ${A}_{q} \subset {A}_{p}$, whence
${\F}_{q} \supset {\F}_{p}$. 

	Next, fix $p, q \in {\P}_{0}$ and suppose that $\phi(p)$ and $\phi(q)$
are compatible in $\FF \ast \Mi(\mathring{\U})$. We must show that $p$ and $q$
are compatible. Indeed, we will prove something stronger. Let $\langle \F,
\langle {s}^{\ast}, d \rangle \rangle$ be an arbitrary tuple where 
	\begin{enumerate}
		\item
			$\F$ is an ${F}_{\sigma}$ filter containing both
${\F}_{p}$ and ${\F}_{q}$ 
		\item
			${s}^{\ast} \in \fin$, ${s}^{\ast} \supset {s}^{p}$,
${s}^{\ast} \supset {s}^{q}$, ${s}^{\ast} \setminus {s}^{p} \subset
\interior(p)$, and ${s}^{\ast} \setminus {s}^{q} \subset \interior(q)$
		\item
			$d \in \F$ and $\forall i \in {s}^{\ast} \forall j \in d
\[i < j\]$
		\item 
			$d \subset \interior(p) \cap \interior(q)$. 
	\end{enumerate}
We will show that there is $r \in {\P}_{0}$ such that $r \leq p$, $r \leq q$,
and $\phi(r) \leq \langle \F, \langle {s}^{\ast}, d \rangle \rangle$. The
argument that $\phi''{\P}_{0}$ is dense in $\FF \ast \Mi(\mathring{\U})$ is
almost identical; so this is enough to finish the proof. Using Lemma
\ref{lem:closed}, find a non-empty closed set $C \subset \Pset(\omega)$ such
that $C \subset \F$ and $\forall b \in \F \exists c \in C \[c \;
{\subset}^{\ast} \; b\]$. Put 
	\begin{align*}
		A = \{s \in \fin: s \in {A}_{p} \cap {A}_{q} \wedge \forall c
\in C \[\lc s \cap c \rc > 1\]\}.
	\end{align*} 
We note a few properties of $A$. It is clear that for each $s \in A$, $\lc s \rc
> 1$ and that if $t \supset s$, then $t \in A$. Next, fix $b \in {\F}^{+}$. For
any $c \in C$, $b \cap c \in {\F}^{+}$. Therefore, there exist $s \in {A}_{p}$
and $\bar{s} \in {A}_{q}$ such that $s \subset b \cap c$ and $\bar{s} \subset b
\cap c$. By a compactness argument, this implies that there is a finite set $s
\subset b$ such that for each $c \in C$, there exists $t \subset s$ such that $t
\in {A}_{p}$ and $t \subset b \cap c$, and also there exists $\bar{t} \subset s$
such that $\bar{t} \in {A}_{q}$ and $\bar{t} \subset b \cap c$. Recall that for
any $t \in {A}_{p}$, $\lc t \rc > 1$. Therefore, for any $c \in C$, $\lc s \cap
c \rc > 1$. Moreover, since $C$ is non-empty, there are $t \subset s$ and
$\bar{t} \subset s$ with $t \in {A}_{p}$ and $\bar{t} \in {A}_{q}$. Therefore,
$s \in {A}_{p} \cap {A}_{q}$. Thus we have shown that for $b \in {\F}^{+}$,
there exists $s \subset b$ such that $s \in A$. Lastly, note that for any 
$c \in C$, there is no $s \in A$ such that $s \subset \left( \omega \setminus c
\right)$.

	Now, let $\nor: \fin \rightarrow \omega$ be the norm induced by $A$,
defined exactly as in the proof of Lemma \ref{lem:main}. It is easy to check
that $\nor$ is well-defined and that it is a norm on $\omega$. Just as in the
proof of Lemma \ref{lem:main}, it is not hard to show by induction on $n$ that
for any $b \in {\F}^{+}$ there exists $s \subset b$ such that $\nor(s) \geq n$.
Define $r$ as follows. ${s}^{r} = {s}^{\ast}$. Let ${n}_{p}$ be the least $n \in
\omega$ such that for all $m \geq n$, ${s}^{p}_{m} \cap {s}^{\ast} = 0$, and let
${n}_{q}$ be analogously defined for $q$. Clearly, $d \cap \left( {\bigcup}_{m
\in [{n}_{p}, \infty)}{{s}^{p}_{m}} \right) \cap \left( {\bigcup}_{m \in
[{n}_{q}, \infty)}{{s}^{q}_{m}}\right) \in \F$. So find ${s}^{r}_{0} \subset d
\cap \left( {\bigcup}_{m \in [{n}_{p}, \infty)}{{s}^{p}_{m}} \right) \cap \left(
{\bigcup}_{m \in [{n}_{q}, \infty)}{{s}^{q}_{m}}\right)$ with $\nor({s}^{r}_{0})
> 0$. Now, suppose that ${s}^{r}_{n}$ is given to us with ${s}^{r}_{n} \subset d
\cap \left( {\bigcup}_{m \in [{n}_{p}, \infty)}{{s}^{p}_{m}} \right) \cap \left(
{\bigcup}_{m \in [{n}_{q}, \infty)}{{s}^{q}_{m}}\right)$ and $\nor({s}^{r}_{n})
> 0$. Put ${n}^{\ast}_{p} = \max\{m \in \omega: {s}^{r}_{n} \cap {s}^{p}_{m}
\neq 0\}$ and ${n}^{\ast}_{q} = \max\{m \in \omega: {s}^{r}_{n} \cap {s}^{q}_{m}
\neq 0\}$. Note that ${n}_{p} \leq {n}^{\ast}_{p}$ and that ${n}_{q} \leq
{n}^{\ast}_{q}$. Again, it is clear that $d \cap \left( {\bigcup}_{m \in
\[{n}^{\ast}_{p} + 1, \infty\right)}{{s}^{p}_{m}} \right) \cap \left(
{\bigcup}_{m \in \[{n}^{\ast}_{q} + 1, \infty\right)}{{s}^{q}_{m}}\right) \in
\F$. So it is possible to find ${s}^{r}_{n + 1}$ with $\nor({s}^{r}_{n + 1}) >
\nor({s}^{r}_{n})$ such that ${s}^{r}_{n + 1} \subset d \cap \left( {\bigcup}_{m
\in \[{n}^{\ast}_{p} + 1, \infty\right)}{{s}^{p}_{m}} \right) \cap \left(
{\bigcup}_{m \in \[{n}^{\ast}_{q} + 1, \infty\right)}{{s}^{q}_{m}}\right)$. This
completes the construction of the ${s}^{r}_{n}$. For each $n \in \omega$, put
${c}^{r}_{n} = \langle {s}^{r}_{n}, \nor \restrict {s}^{r}_{n} \rangle$ and define $r = \langle {s}^{r}, \langle {c}^{r}_{n}: n \in \omega \rangle \rangle$. Observe that for any $s \in \fin$, if $\nor(s) > 0$, then $s \in A$, and so $s \in  {A}_{p} \cap {A}_{q}$, and hence there exist $m, n \in \omega$ such that ${\nor}^{p}_{m}(s
\cap {s}^{p}_{m}) > 0$ and ${\nor}^{q}_{n}(s \cap {s}^{q}_{n}) > 0$. It follows
that $r \leq p$ and $r \leq q$. It remains to be seen that $\phi(r) = \langle
{\F}_{r}, \langle {s}^{r}, \interior(r) \rangle \rangle \leq  \langle \F,
\langle {s}^{\ast}, d \rangle \rangle$. First suppose that $s \in {A}_{r}$. Then
by definition, for some $n \in \omega$, ${\nor}^{r}_{n}(s \cap {s}^{r}_{n}) >
0$. Hence $s \in A$, and so ${A}_{r} \subset A$. So for any $c \in C$, $\neg
\exists s \in {A}_{r}$ such that $s \subset \omega \setminus c$. So $c \in
{C}_{r}$. Thus $C \subset {C}_{r}$. It follows that $\F \subset {\F}_{r}$. Since
${s}^{r} = {s}^{\ast}$ and $\interior(r) \subset d$, it follows that $\phi(r)
\leq \langle \F, \langle {s}^{\ast}, d \rangle \rangle$.        
\end{proof}
We make some remarks on how to get an analogous characterization for ${\P}_{1}$.
Let $\A$ be as in Section \ref{sec:blessthanaclosed}.
Let ${\V}_{{\omega}_{1}}$ be the extension gotten by adding ${\omega}_{1}$ Cohen
reals.
Then in ${\V}_{{\omega}_{1}}$ it is possible to prove that ${\P}_{1}$ (defined
relative to $\A$) densely embeds into ${\FF}_{\A}\ast\MM(\mathring{\U})$, where
${\FF}_{\A} = \{\F: \F \ \text{is a proper} \ {F}_{\sigma} \ \text{filter on} \
\omega \ \text{and} \ \I(\A) \cap \F = 0\}$, ordered by $\supset$, and where
$\mathring{\U}$ is the canonical ${\FF}_{\A}$ name for the ultrafilter added by
it.
The proof of this is nearly identical to the proof of Theorem
\ref{thm:p0isfstarmu}, expect that in the construction of $r$, the Cohen reals
must be used like in the proof of Lemma \ref{lem:main}.
Now, it is easy to see that both in the case of ${\P}_{0}$ and in the case of
${\P}_{1}$, for the corresponding $\MM(\mathring{\U})$ to have the right
properties, it is not necessary for $\mathring{\U}$ to be fully generic for
$\FF$ or ${\FF}_{\A}$ respectively.
It is enough to have ultrafilters that are sufficiently generic for $\FF$ and
${\FF}_{\A}$.
We elaborate on this idea in the next section to give a c.c.c.\ proof of the
consistency of $\b < \ac$. 

\section{A ccc proof} \label{sec:cccproof}

In this section, we provide a ccc proof of the consistency of $\bb <
\aa_\closed$. Unlike the proof in
Section~\ref{sec:blessthanaclosed}, this proof generalizes to the situation
where $\cc$ is larger than
$\omega_2$.

Let $\kappa$ be a regular uncountable cardinal, assume $\cc = \kappa$, $\la
f_\alpha : \alpha < \kappa \ra$
is a well-ordered unbounded family in $\omom$, 
and $\la X_\alpha : \alpha < \lambda \ra$ is a sequence of non-empty closed
subsets of $\omoms$ such that $\A = \bigcup_{\alpha < \lambda} X_\alpha$ is a
MAD family. Here
$\omega_1 \leq \lambda \leq \kappa$. Let $\VVV_\kappa$ be the extension of
$\VVV$ by adding $\kappa$
Cohen reals. Assume that (the reinterpretation of) $\A$ is still MAD in
$\VVV_\kappa$.

\begin{Theorem}
There is an ultrafilter $\U$ extending $\F (\A)$ such that $\MM (\U)$ preserves
the unboundedness of
$\la f_\alpha : \alpha < \kappa \ra$ and forces that (the reinterpretation of)
$\A$ is not MAD anymore.
\end{Theorem}

\begin{proof}
The proof of the theorem follows closely the proof of the analogous result for
$\aa$ instead of $\aa_\closed$,
\cite[Theorem 3.1]{Br98}. 
However, some of the combinatorics developed for $\aa_\closed$ in
Section~\ref{sec:blessthanaclosed}
will be needed as well.

Say $\F$ is an $F_{<\kappa}$ filter if it is the union of $<\kappa$ many closed
subsets of $\omoms$.
It is easy to see that the appropriate generalizations of Lemmas~\ref{lem:p}
and~\ref{lem:ad}
hold. 

\begin{Lemma}   \label{kappa1}
In $\VVV_\kappa$, let $\F$ be any $F_{<\kappa}$ filter and suppose that $\GG$,
the filter generated by 
$\F \cup \F(\A)$, is a proper filter. Then $\GG$ is $P^+$.
\end{Lemma}

\begin{Lemma}  \label{kappa2}
In $\VVV_\kappa$, suppose that $\F$ is a $F_{<\kappa}$ filter such that $\GG$,
the filter generated by
$\F \cup \F(\A)$, is proper. Suppose $b \in \GG^+$. Then for each $\alpha_0, ...
, \alpha_k < \omega_1$,
there is a $c \in [b]^\omega$ such that $c \in \GG^+$ and $\forall (a_0, ... ,
a_k) \in X_{\alpha_0} \times ...
\times X_{\alpha_k} [|(a_0 \cup ... \cup a_k) \cap c | < \omega]$.
\end{Lemma}

We distinguish two cases. They correspond to the cases where we force with the
partial orders
$\P_0$ and $\P_1$, respectively, in Section~\ref{sec:blessthanaclosed}, and also
to the two cases
of the proof of~\cite[Theorem 3.1]{Br98}.

{\em Case 1.} In $\VVV_\kappa$, there is a $F_{<\kappa}$ filter $\F$ such that
$\F (\A) \sub \F$. This corresponds
to the situation where we force with $\P_0$ in
Section~\ref{sec:blessthanaclosed}. Since this case
is different from the corresponding case in~\cite{Br98}, we provide details.
Recall~\cite[p. 192]{Br98} that a partial map $\tau : \omloms \times \omega \to
\omega$ is a {\em preterm}.
If $\GG \supseteq \F$ is a filter and $\mathring g$ is an $\MM (\GG)$-name for a
function in $\omom$, then
$\tau = \tau_{\mathring g}$ given by $\tau (s,n) = k$ iff $(s,G)$ forces
``$\mathring g (n) = k$" for some $G \in \GG$
is a preterm, the {\em preterm associated with $\mathring g$}. Let $\{
\tau_\alpha : \alpha < \kappa \}$
enumerate the set of all preterms. Let $\F_0 = \F$. Recursively build an
increasing chain of $F_{<\kappa}$ filters 
$\F_\alpha$, $\alpha < \kappa$, such that
\begin{itemize}
\item for all $\alpha < \lambda$ there is $b \in \F_{\alpha +1}$ such that $| b
\cap a| < \omega$ for all $a \in X_\alpha$
\item if $\tau_\alpha$ looks like a name for $\F_\alpha$, then there is $\beta <
\kappa$ such that for all
  filters $\HHH$ extending $\F_{\alpha + 1}$, $\MM (\HHH)$ forces that $f_\beta
\not\leq^* \mathring g$ where
  $\tau_{\mathring g} = \tau_\alpha$
\item if $\tau_\alpha$ does not look like a name for $\F_\alpha$, then
$\tau_\alpha$ is not a preterm
  associated with any $\MM (\HHH)$-name $\mathring g$, for any filter $\HHH$
extending $\F_{\alpha + 1}$.
\end{itemize}
Here we say that $\tau_\alpha$ {\em looks like a name for} $\F_\alpha$ if for
all $n$, all
$s \in \omloms$, and all $c \in \F_\alpha^+$, there are $t \in \omloms$ and $u
\sub s \cup t$
such that $(u,n) \in \dom (\tau_\alpha)$ and $t \sub c$. 

For limit ordinals $\alpha$ we simply let $\F_\alpha = \bigcup_{\beta < \alpha}
\F_\beta$. So assume
$\alpha + 1$ is a successor ordinal. Suppose $\tau_\alpha$ does not look like a
name for $\F_\alpha$.
Then there are $n$, $s \in \omloms$, and $c \in \F_\alpha^+$ witnessing this.
That is, whenever
$t \sub c$ is finite, then for no $u \sub s \cup t$ does $(u,n)$ belong to
$\dom(\tau_\alpha)$. 
Let $\F_\alpha' $ be the filter generated by $\F_\alpha$ and $c$. Then
$\tau_\alpha$ is not a preterm
associated with any $\MM (\HHH)$-name $\mathring g$, for any filter $\HHH$
extending $\F_\alpha '$,
because no condition compatible with $(s,c)$ would decide $\mathring g (n)$.

So suppose $\tau_\alpha$ looks like a name for $\F_\alpha$. Assume $\F_\alpha =
\bigcup_{\gamma
< \mu} K_\gamma$ where $\mu < \kappa$ and all $K_\gamma$ are compact. Fix
$\gamma$
and fix $T = \{ s_j : j < \ell \} \sub \omloms$. Now define $f = f_{\gamma, T}$
by
$$\begin{array}{ll} f(n) = & \min \{ k : \mbox{ given } c \in K_\gamma \mbox{
and } b_j \mbox{ with }
  c \sub \bigcup_{j < \ell} b_j \\ & \mbox{there are } j <\ell , t \subset b_j,
\mbox{ and } u \sub s_j \cup t 
  \mbox{ with } \tau_\alpha (u,n) \leq k \} \\
\end{array}$$
Let us first check that $f$ is well-defined. Fix $c \in K_\gamma$ and $b_j$ with
$c \sub \bigcup_{j < \ell}
b_j$. Let $j$ be minimal such that $b_j \in \F_\alpha^+$. Since $\tau_\alpha$
looks like a name for
$\F_\alpha$, there are finite $t \sub b_j$ and $u \sub s_j \cup t$ such that
$(u,n) \in \dom (\tau_\alpha)$.
Choose such $t$ and $u$ so that the value $k ( \{ c, b_j : j < \ell \}) :=
\tau_\alpha (u,n)$ is minimal.
Since $(\twoom)^\ell$ and $K_\gamma$ are compact, it is easy to see that the
function sending
$\la c, b_j : j < \ell\ra$ to  $k ( \{ c, b_j : j < \ell \}) $ is bounded. Hence
$f$ is well-defined.

Now choose $\beta$ such that $f_\beta \not\leq^* f_{\gamma ,T}$ for all $\gamma
< \mu$ and
finite $T \sub \omloms$. For $s \in \omloms$ and $b \in\omoms$ define $g =
g_{s,b}$ by
$$ g(n) =  \min \{ k : \exists \mbox{ finite } t \sub b \; \exists u \sub s \cup
t \; ( \tau_\alpha (u,n) = k ) \} $$
in case the set on the right-hand side is non-empty; otherwise put $g(n)
=\omega$.
Let $\F_\alpha '$ be the filter generated by $\F_\alpha$ and all sets of the
form
$\{\omega \sem b : \exists s \; (g_{s,b} \geq^* f_\beta ) \}$. It is clear that
these
sets are a union of countably many compact sets. 

We first verify that $\F_\alpha '$ still is a proper filter. Suppose this were
not the case.
Then, for $c \in \F_\alpha$ and sets $b_j$, $j < \ell$, we would have $\omega
\sem b_j \in \F_\alpha '$
and $c \cap \bigcap_{j < \ell} \omega \sem b_j = \emptyset$, i.e., $c \sub
\bigcup_{j<\ell} b_j$.
Fix $\gamma$ such that $c \in K_\gamma$ and $s_j$ such that $g_{s_j, b_j}
\geq^*f_\beta$.
Set $T = \{ s_j : j < \ell \}$. Fix $m$ such that $g_{s_j,b_j} (n) \geq
f_\beta(n) $ for all $n \geq m$.
By construction there is $n \geq m$ such that $f_{\gamma, T} (n) < f_\beta (n)$.
By definition
of $f_{\gamma, T}$, there are $j < \ell$, $t \sub b_j$, and $u \sub s_j \cup t$
with 
$\tau_\alpha (u,n) \leq f_{\gamma,T} (n)$. But then $g_{s_j,b_j} (n) \leq
f_{\gamma ,T}
(n) < f_\beta (n) \leq g_{s_j,b_j} (n)$, a contradiction.

Next we check that $\F_\alpha '$ is as required. Let $\HHH$ be any filter
extending $\F_\alpha'$,
and let $(s,b) \in \MM (\HHH)$. Suppose $\mathring g$ is $\MM (\HHH)$-name such
that
$\tau_\alpha = \tau_{\mathring g}$. Assume there is $m$ such that $(s,b)$ forces
$\mathring g(n) \geq f_\beta (n)$ for all $n \geq m$. Then clearly $g_{s,b} (n)
\geq f_\beta (n)$ for
all $n \geq m$. So $\omega \sem b \in \F_\alpha ' \sub \HHH$, a contradiction.

Finally, by Lemma~\ref{kappa2}, we may find $b \in (\F_\alpha')^+$ such that
$|b \cap a | < \omega$ for all $a \in X_\alpha$. Let $\F_{\alpha+1}$ be the
filter generated by 
$\F_\alpha'$ and $b$. This completes the recursive construction and Case 1 of
the proof.

{\em Case 2.} In $\VVV_\kappa$, there is no $F_{<\kappa}$ filter $\F$ such that
$\F (\A) \sub \F$.
This corresponds to the situation when we force with $\P_1$ in Section 1.
This is the more difficult case. However, unlike for Case 1, the proof
of~\cite{Br98} can be taken
over almost verbatim in this case. Simply mix applications of Lemma~\ref{kappa2}
with the recursive construction expounded in \cite[pp. 192-195]{Br98}.

This completes the proof of the theorem.
\end{proof}

Using finite support iteration we now obtain
\begin{Theorem}
Let $\kappa$ be a regular uncountable cardinal. It is consistent that $\bb \leq
\kappa$
and $\aa_\closed = \cc = \kappa^+$.
\end{Theorem}


\section{Tail splitting, club splitting and closed almost disjointness}
\label{sec:club}

\begin{Def} \label{def:clubandtail}
Let $\kappa$ be a regular cardinal, and let $\bar A = \la a_{\alpha} : \alpha <
\kappa \ra \sub \omoms$.
$\bar A$ is {\em tail-splitting} if for every $b \in \omoms$ there is $\alpha <
\kappa$ such that
$a_\beta $ splits $b$ for all $\beta \geq \alpha$. $\bar A$ is {\em
club-splitting} if for every $b \in \omoms$,
$C_b = \{ \alpha < \kappa : a_\alpha$ splits $b\}$ contains a club. 
\end{Def}
Clearly, a tail-splitting sequence
is club-splitting, and the existence of a club-splitting sequence of length
$\kappa$ implies that
$\ss_\omega \leq \kappa$. Moreover, it is easy to see that $\kappa \leq \rr$,
where $\rr$ is the reaping number. In the next section we shall come back to the
question which
of these implications reverse.

\begin{Def}
$\bar A = \la a_{\alpha,n} : \alpha < \kappa, n < \omega \ra$ is a {\em
tail-splitting
sequence of partitions} if the $a_{\alpha,n}$, $n \in \omega$, are pairwise
disjoint and for
all $b \in \omoms$ there is $\alpha$ such that $a_{\beta,n}$ splits $b$ for all
$\beta \geq \alpha$
and all $n \in \omega$. Similarly, $\bar A$ is a {\em club-splitting sequence of
partitions} if 
for all $b \in \omoms$, $C_b = \{ \alpha < \kappa : $ all $a_{\alpha ,n}$ split
$b\}$ contains a
club. 
\end{Def}
Clearly a tail-splitting sequence of partitions yields a tail-splitting
sequence, but we don't know
whether the converse is true (see Question~\ref{splitting-partitions}).
Similarly for club-splitting.

We begin with two observations:

\begin{Obs}
In the Hechler model (the model obtained by adding at least $\omega_2$ Hechler
reals
over a model of $CH$), there is a tail-splitting sequence of partitions of
length $\omega_1$.
\end{Obs}

To see this notice that the classical proof, of the consistency of $\ss < \bb$,
due to Baumgartner and Dordal~\cite{BD85},
shows that tail-splitting sequences of partitions from the ground model are
preserved
in the iterated Hechler extension.

\begin{Obs}
$\dd = \aleph_1$ implies the existence of a tail-splitting sequence of
partitions
of length $\omega_1$.
\end{Obs}

\begin{Def} \label{splitting-over-models-def}
Say there is a {\em splitting sequence of partitions over models} if there are 
$\bar M = \la M_\alpha : \alpha < \omega_1 \ra$ and $\bar A = \la a_{\alpha,n} :
\alpha < \omega_1, n < \omega \ra$ such that 
\begin{itemize}
\item $\bar M$ is a strictly increasing continuous sequence of countable models
of a 
  large enough fragment of ZFC
\item for each $\alpha$, $\la a_{\alpha,n}: n\in\omega\ra$
  is pairwise disjoint, belongs to $M_{\alpha+1}$, and all $a_{\alpha,n}$ split
all members
  of $M_\alpha$
\item whenever $b \in \omoms$, there are $\alpha$ and a model $N$ of a large
enough fragment
  of ZFC containing $b$ such that $M_\alpha \sub N$, $N \cap M = {M}_{\alpha}$,
and 
  all $a_{\alpha,n}$ split all members of $N$.
\end{itemize}
Here, $M = \bigcup_{\alpha < \omega_1}  M_\alpha$.
\end{Def} 

\begin{Lemma}\label{splitting-over-models}
The existence of a club-splitting sequence of partitions of length $\omega_1$
implies
the existence of a splitting sequence of partitions over models.
\end{Lemma}

\begin{proof}
Assume $\bar B = \la b_{\alpha,n} : \alpha < \omega_1, n < \omega \ra$ is a 
club-splitting sequence of partitions.
Let $\chi$ be a large enough regular cardinal.
Let $\bar{M} = \langle {M}_{\alpha}: \alpha < {\omega}_{1} \rangle$ be such that
for each $\alpha < {\omega}_{1}$
	\begin{enumerate}
		\item
			$\bar{B} \in {M}_{0}$, ${M}_{\alpha} \< H(\chi)$, $\lc
{M}_{\alpha} \rc = \omega$, and ${M}_{\alpha} \in {M}_{\alpha + 1}$.
		\item
			if $\alpha$ is a limit, then ${M}_{\alpha} =
{\bigcup}_{\xi < \alpha}{{M}_{\xi}}$. 
	\end{enumerate}
For each $\alpha < {\omega}_{1}$, let ${\delta}_{\alpha} = {M}_{\alpha} \cap
{\omega}_{1}$.
Define $\langle {a}_{\alpha, n}: n \in \omega \rangle = \langle
{b}_{{\delta}_{\alpha}, n}: n \in \omega \rangle$.
For any $\alpha < {\omega}_{1}$ and $x \in \cube \cap {M}_{\alpha}$, there is a
club $C \in {M}_{\alpha}$ such that for all $\delta \in C$ and $n \in \omega$,
${b}_{\delta, n}$ splits $x$.
As ${\delta}_{\alpha} \in C$, ${a}_{\alpha, n}$ splits $x$ for all $n \in
\omega$.
Next, if $b \in \cube$, then let $N \< H(\chi)$ be countable with $\bar{M} \in
N$ and $b \in N$.
Let $\gamma = N \cap {\omega}_{1}$.
It is clear that $N \cap \left( {\bigcup}_{\xi < {\omega}_{1}}{{M}_{\xi}}\right)
= {M}_{\gamma}$ and moreover, $\gamma = {\delta}_{\gamma}$.
Again, for any $x \in \cube \cap N$ there is a club $C \in N$ such that for all
$\delta \in C$ and $n \in \omega$, ${b}_{\delta, n}$ splits $x$.
As $\gamma = {\delta}_{\gamma} \in C$, we are done. 
\end{proof}

\begin{Theorem}\label{splitting-over-models-aclosed}
The existence of a splitting sequence of partitions over models implies
$\aa_\closed = \aleph_1$.
\end{Theorem}

\begin{proof}
This follows from a straightforward analysis of the proof of~\cite[Lemma
3.4]{BKta}.
Since the proof of the latter lemma is rather long and technical, we will not
repeat it here
and simply stress the main points. We assume the reader to have a copy
of~\cite{BKta}
at hand.

Assume we are at stage $\alpha$, and closed sets $A_\beta \in M_\alpha$ have
been 
constructed so that $\bigcup_{\beta<\alpha} A_\beta$ is an almost disjoint
family. (We do not assume
that the whole sequence of the $A_\beta$ belongs to $M_\alpha$; this does not
matter.)
The $A_\beta$ are obtained as sets of branches through a tree whose levels form
a partition of a subset of $\omega$. Now, from the $a_{\alpha,n}$, one obtains 
a sequence $C_\sigma^\Theta$ of pairwise disjoint subsets of $\omega$, where
$\sigma \in \omlom$ and $\Theta$ comes from a certain set of finite sequences of
finite
sequences, which is used to construct the next set $A_\alpha$. To obtain the
$C_\sigma^\Theta$
from the $a_{\alpha,n}$, one has to remove finitely many elements (the
``excluded points")
as well as a set from $M_\alpha$ (the set $X_\sigma$), see the end of part 1 in
the
proof of~\cite[Lemma 3.4]{BKta} for details. Obviously, the resulting
$C_\sigma^\Theta$ will still
split all $Y \in M_\alpha$ such that $Y \sem X_\sigma$ is infinite, and this is
all that's
needed for the rest of the proof to go through. This completes the construction
of
the $A_\alpha$. We need to check they are as required.

Part 2 of the proof of~\cite[Lemma 3.4]{BKta} does not apply, 
and steps 1 and 2 of part 3 carry over without any change.
The heart of the proof is step 3 of part 3 (the last part of the proof), namely,
the argument showing that $\bigcup_{\beta< \omega_1} A_\beta$ is indeed
maximal. Take any $Y \in \omom$. Find
$\alpha$ and $N$ such that they satisfy the last clause of
Definition~\ref{splitting-over-models-def}
for $b = Y$.
Now, as in the proof of~\cite[Lemma 3.4]{BKta}, build functions $g_j \in \omom
\cap N$ and
a decreasing sequence of subsets $Y_j \in N$ of $Y$. This is possible because
$M_\alpha \sub N$. (Again, we do not require that
the sequences of the $g_j$ or $Y_j$ belong to $N$, but this is not needed.)
Assume that $Y$ is almost disjoint from all elements of $A_\beta$, for $\beta <
\alpha$.
Using the $g_j$ and $Y_j$ a function $h$ is constructed such that the branch in
$A_\alpha$
associated with $h$ is a subset of $Y$, i.e. there is $a \in A_\alpha$ with $a
\sub Y$.
For the construction of $h$, the splitting properties of the $C^\Theta_\sigma$
together
with the fact that any initial segment of $h$ is constructed in $N$ are used.
\end{proof}

Using the theorem, we obtain two results from the literature as corollaries.

\begin{Cor}[Brendle and Khomskii, \cite{BKta}]
In the Hechler model, $\aa_\closed = \aleph_1$. In particular, $\bb >
\aa_\closed$ is consistent.
\end{Cor}

\begin{Cor}[Raghavan and Shelah, \cite{RS12}]
$\dd = \aleph_1$ implies $\aa_\closed = \aleph_1$.
\end{Cor}


\section{Tail splitting: a consistency result} \label{sec:clubandtail}

In this section, we show that tail-splitting and club-splitting are not the
same.

\begin{Theorem} \label{tail-club}
It is consistent that there is a club-splitting family of size $\aleph_1$
and there is no tail-splitting family of size $\aleph_1$.  In particular, $\sss
= \aleph_1$.
\end{Theorem}

Assume $\bar A = \la a_\alpha : \alpha < \omega_1\ra$ is club-splitting. Let
$\P$ be a forcing notion.
Say that {$\P$ preserves $\bar A$} if $\bar A$ is still club-splitting in the
$\P$-generic extension.
It is easy to see that if $(\P_\alpha : \alpha < \delta)$ is an fsi of ccc
forcing and all
$\P_\alpha$ ($\alpha < \delta$) preserve $\bar A$, then so does $\P_\delta$.

Also let $\HHH$ be a filter on $\omega$.
We say that $(\star)_{\bar A,\HHH}$ holds if for every partial function $f :
\omega \to \omega$
with $\dom (f) \in \HHH^+$ and $f^{-1} ( \{ n \} ) \in \HHH^*$, the set $D_f =
\{ \alpha < \omega_1 : f^{-1} (a_\alpha)$ and
$f^{-1} (\omega \sem a_\alpha)$ both belong to $\HHH^+ \}$ contains a club.

\begin{Lemma}
Assume $(\star)_{\bar A,\HHH}$ holds. Then $\LL (\HHH)$ preserves $\bar A$.
\end{Lemma}

\begin{proof}
Let $\mathring a$ be an $\LL(\HHH)$-name for an infinite subset of $\omega$. We
need to find a
club set $C \subset \omega_1$ in the ground model such that the trivial
condition forces
that $a_\alpha$ splits $\mathring a$ for all $\alpha \in C$. We can assume that
$\mathring a$ is thin in the
sense that the increasing enumeration $\mathring g$ of $\mathring a$ is forced
to dominate the generic
Laver real $\mathring \ell$. 

We briefly recall the standard rank analysis of Laver forcing $\LL (\HHH)$. Let
$\varphi$ be
a formula. For any $s \in \omlom$, say that $s$ {\em forces} $\varphi$ if there
is a condition
with stem $s$ which forces $\varphi$. Say that $s$ {\em favors} $\varphi$ if $s$
does not
force $\neg \varphi$. Define the rank function $\rk_\varphi$ by induction:
\begin{itemize}
\item $\rk_\varphi (s) = 0$ iff $s$ forces $\varphi$
\item $\rk_\varphi (s) \leq \alpha$ iff there is $c \in \HHH^+$ such that
$\rk_\varphi (s\ha{}n) < \alpha$
  for all $n \in c$
\item $\rk_\varphi (s) = \alpha$ iff $\rk_\varphi (s) \leq \alpha$ but
$\rk_\varphi (s) \not\leq \beta$ for $\beta < \alpha$.
\end{itemize}
A standard argument shows that $s$ favors $\varphi$ iff $\rk_\varphi (s) <
\omega_1$.
(Suppose $\rk_\varphi (s)$ is undefined. Then one constructs a tree $T \in \LL
(\HHH)$ with stem $s$
such that for all nodes $t \in T$ extending $s$, $\rk_\varphi (t)$ is undefined.
In particular, no extension
of $s$ in $T$ has rank $0$, and therefore $T$ must force $\neg \varphi$. Thus
$s$ does not favor $\varphi$.
Suppose, on the other hand, that $s$ forces $\neg\varphi$. We prove by induction
on $\alpha$ that
$\rk_\varphi (s) > \alpha$. This is obvious for $\alpha = 0$. So assume $\alpha
> 0$. Let
$T \in \LL (\HHH)$ be a tree with stem $s$ witnessing that $s$ forces
$\neg\varphi$. Let $c \in \HHH$
be the successor level of $s$ in $T$. By induction hypothesis $\rk_\varphi
(s\ha{} n) \geq\alpha$
for all $n \in c$. By definition of the rank, we see that $\rk_\varphi (s) >
\alpha$.)

Say that $s \in \omlom$ is {\em good for} $n$ if $s$ does not favor $\mathring g
(n) = k$ for any $k$,
but $\{ m : s\ha{} m$ favors $\mathring g (n) =k$ for some $k\}$ is
$\HHH$-positive.
\begin{ssclaim} 
If $|s| \leq n$ and $\stem (T) = s$, then there is $t \in T$ extending $s$ which
is good for $n$.
\end{ssclaim}

\begin{proof}
Define a new rank function $\rho$ by stipulating 
\begin{itemize}
\item $\rho (t) = 0$ if $t$ favors $\mathring g(n) = k$ for some $k$
\item $\rho (t) \leq \alpha$ iff there is $c \in \HHH^+$ such that $\rho
(t\ha{}n) < \alpha$
  for all $n \in c$.
\end{itemize}
Notice that $\rho (s) < \omega_1$. (Otherwise there would be a tree $T' \in \LL
(\HHH)$ with stem $s$ such that
all nodes of $T'$ extending $s$ have undefined rank. Now find $t \in T'$
extending $s$ and forcing $\mathring g(n) = k$
for some $k$. Clearly $\rho (t) =0$, a contradiction.) On the other hand, $|s|
\leq n$ and $\mathring g \geq \mathring\ell$
imply that $\rho(s) \geq 1$ because for each $k$ there is a tree $T'$ with stem
$s$ forcing $\mathring \ell (n) > k$ and,
hence, $\mathring g (n) > k$. Thus we can find $t \in T$ extending $s$ such that
$\rho (t) = 1$. By definition,
this means that $t$ does not favor $\mathring g(n) = k$ for any $k$, and that
$\{ m : t\ha{} m$ favors $\mathring g (n) = k$
for some $k \}$ belongs to $\HHH^+$.
\end{proof}

For each node $s$ which is good for $n$, define a partial function $f_{s,n}$ by
letting $\dom (f_{s,n}) =
\{ m : s\ha{} m$ favors $\mathring g (n) =k$ for some $k\}$ and setting $f_{s,n}
(m) = k$ for some $k$ such that
$s\ha{} m$ favors $\mathring g (n) =k$, for $m \in \dom (f_{s,n})$. Note that
such $k$ is not necessarily unique,
but this does not matter. By definition of goodness, it is immediate that
$f_{s,n}$ satisfies the stipulations
in the definition of $(\star)_{\bar A,\HHH}$, i.e. $\dom (f_{s,n} ) \in \HHH^+$
and $f_{s,n}^{-1} (\{k\}) \in  \HHH^*$
for all $k$. Now let $C$ be the intersection of all $D_{f_{s,n}}$ where $s$ is
good for $n$. We show that
$C$ is as required. 

\begin{ssclaim}
The trivial condition forces that $a_\alpha$ splits $\mathring a$ for all
$\alpha \in C$.
\end{ssclaim}

\begin{proof}
Let $T$ be any condition and $n_0$ a natural number. We need to find $n, n' \geq
n_0$ and $T'\leq T$
such that $T' $ forces $\mathring g (n) \in a_\alpha$ and $\mathring g(n')
\notin a_\alpha$. Since the proofs are identical,
we only produce $n$. Let $s$ be the stem of $T$. Choose $n \geq n_0, |s|$. By
the previous claim,
there is $t\in T$ extending $s$ which is good for $n$. Hence $f_{t,n}$ is
defined. Since $\alpha \in 
D_{f_{t,n}}$, $f^{-1}_{t,n} (a_\alpha)$ belongs to $\HHH^+$. Hence we can find
$m \in \dom (f_{t,n})$ in the
successor level of $t$ in $T$ such that $k:=f_{t,n} (m) \in a_\alpha$. Since
$t\ha{} m$ favors $\mathring g (n) = k$,
there is a subtree $T'$ of $T$ with stem extending $t\ha{}m$ which forces
$\mathring g(n) = k$.
Therefore $T'$ forces $\mathring g (n) \in a_\alpha$, as required.
\end{proof}

This completes the proof of the lemma.
\end{proof}

\begin{Lemma}
Assume CH. Assume $\bar B =\la b_\alpha : \alpha < \omega_1 \ra$ is
tail-splitting. Then there is
$\{ c_\alpha : \alpha < \omega_1 \}$ such that $c_\alpha \sub b_{\zeta_\alpha}$
for some
$\zeta_\alpha \geq \alpha$ 
and the $c_\alpha$ generate a P-filter $\HHH$ such that $(\star)_{\bar A,\HHH}$
holds.
\end{Lemma}

\begin{proof}
Let $\{ f_\alpha : \alpha < \omega_1 \}$ list all partial finite-to-one
functions $\omega \to \omega$.
Recursively we find $\sub^*$-decreasing $c_\alpha \in \omoms$, $\zeta_\alpha
\geq \alpha$,
continuous increasing $\gamma_\alpha$, and decreasing club sets $C_\alpha$ such
that
\begin{itemize}
\item $c_\alpha \sub^* b_{\zeta_\alpha}$ 
\item $\gamma_\alpha \in C_\alpha$
\end{itemize}
and for all $\beta < \alpha$ such that $\dom (f_\beta) \cap c_\beta$ is
infinite,
\begin{itemize}
\item $a_{\gamma_\delta}$ splits $f_\beta (c_\alpha)$ (i.e.
  $f_\beta^{-1} (a_{\gamma_\delta}) \cap c_\alpha$ and $f_\beta^{-1} (\omega
\sem a_{\gamma_\delta}) \cap c_\alpha$
  are both infinite) for $\beta < \delta \leq \alpha$
\item for all $\gamma \in C_\alpha$, $a_\gamma$ splits the sets $f_\beta
(c_\alpha)$.
\end{itemize}

Basic step: $c_0 = b_0$, $\zeta_0 = 0$, $\gamma_0 = 0$.

Successor step: $\alpha \to \alpha + 1$. Since $\bar B$ is tail-splitting, we
can find $\zeta_{\alpha +1} \geq \alpha + 1$ 
such that $b_{\zeta_{\alpha +1}}$ splits all sets $f_\beta^{-1}
(a_{\gamma_\delta}) \cap c_\alpha$
and $f_\beta^{-1} (\omega \sem a_{\gamma_\delta}) \cap c_\alpha$ for $\beta <
\delta \leq \alpha$,
as well as $\dom (f_\alpha) \cap c_\alpha$ if the latter set is infinite.
In particular, the intersection of $b_{\zeta_{\alpha +1}}$ with these sets is
infinite. Let $c_{\alpha + 1}
= c_\alpha \cap b_{\zeta_{\alpha +1}}$. Then $a_{\gamma_\delta}$ splits $f_\beta
(c_{\alpha +1})$
for $\beta < \delta < \alpha +1$. Since $\bar A$ is club-splitting, there is a
club set $C_{\alpha +1} \sub C_\alpha$
such that for all $\gamma \in C_{\alpha +1}$, $a_\gamma$ splits all sets
$f_\beta (c_{\alpha+1})$ for
$\beta < \alpha$, as well as $f_\alpha (c_{\alpha +1})$ in case $\dom (f_\alpha)
\cap c_\alpha$ is infinite.
Now let $\gamma_{\alpha +1}$ be the least element of $C_{\alpha +1}$ greater
than $\gamma_\alpha$.

Limit step: $\alpha$ limit. Let $C' =\bigcap \{ C_\beta : \beta < \alpha \}$. 
Let $\gamma_\alpha = \bigcup \{ \gamma_\beta : \beta < \alpha \}$. Clearly
$\gamma_\alpha \in C '$.
So $a_{\gamma_\alpha}$ splits all $f_\beta (c_\delta)$ where $\beta < \delta <
\alpha$. 
Construct $c'$ as a pseudo-intersection of $c_\delta$, $\delta < \alpha$, such
that
all $a_{\gamma_\delta}$ still split all $f_\beta (c')$ for $\beta < \delta \leq
\alpha$.

Since $\bar B$ is tail-splitting, we can find $\zeta_\alpha \geq \alpha$ such
that $b_{\zeta_\alpha}$ splits
all sets $f_\beta^{-1} (a_{\gamma_\delta}) \cap c'$
and $f_\beta^{-1} (\omega \sem a_{\gamma_\delta}) \cap c'$ for $\beta < \delta
\leq \alpha$.
Let $c_\alpha = c' \cap b_{\zeta_\alpha}$. 
Since $\bar A$ is club-splitting, we can find $C_\alpha \sub C'$ club with
$\gamma_\alpha \in C_\alpha$ and such that
for all $\gamma \in C_\alpha$, $a_\gamma$ splits the sets $f_\beta (c_\alpha)$.

This completes the recursive construction. We need to show that the $c_\alpha$
are as required.
Clearly, they generate a P-filter $\HHH$. Let $f: \omega \to \omega$ be a
partial function with
$\dom (f) \in \HHH^+$ and $f^{-1} (\{ n\}) \in \HHH^*$ for all $n$. Since $\HHH$
is a P-filter,
the sets $f^{-1} (\omega \sem n)$ have a pseudo-intersection $A \in \HHH$. Notice
that the restriction
of $f$ to $A$ is finite-to-one. So we may assume without loss of generality that
$f$ is finite-to-one.
Hence there is $\beta$ such that $f = f_\beta$. Since $\dom (f_\beta) \in
\HHH^+$, 
$\dom (f_\beta) \cap c_\beta$ is clearly infinite. By construction, for all
$\alpha > \beta$ and all
$\delta > \beta$, $a_{\gamma_\delta}$ splits $f_\beta (c_\alpha)$. Hence both
$f_\beta^{-1} (a_{\gamma_\delta})$
and $f_\beta^{-1} (\omega \sem a_{\gamma_\delta})$ are $\HHH$-positive. 
Thus the club set $D_f = D_{f_\beta} = \{ \gamma_\delta : \delta > \beta \}$ is
as required.
\end{proof}

We finally discuss an application of tail-splitting.

\begin{Def}
The {\em strong polarized partition relation} $\left(\begin{array}{l}\lambda \\
\kappa \\ \end{array}\right)
\to \left(\begin{array}{l}\lambda \\ \kappa \\ \end{array}\right)^{1,1}_2$ means
that
for every function $c : \lambda \times \kappa \to 2$ there are $A \sub \lambda$
and $B \sub \kappa$
of size $\lambda$ and $\kappa$, respectively, such that $c \re (A \times B)$ is
constant. 
\end{Def}
The following
was essentially observed by Garti and Shelah~\cite[Claim 1.3]{GSta1}, though
they stated this
in somewhat different language.

\begin{Obs}
The following are equivalent.
\begin{enumerate}
\item $\left(\begin{array}{l}\lambda \\ \omega \\ \end{array}\right)
  \to \left(\begin{array}{l}\lambda \\ \omega \\ \end{array}\right)^{1,1}_2$
\item $cf (\lambda) \neq \omega$ and there does not exist a tail-splitting
sequence of length $\lambda$.
\end{enumerate}
\end{Obs}

In particular, Garti and Shelah~\cite[Claim 1.4]{GS12} observed that $\ss >
\aleph_1$ implies that
$\left(\begin{array}{l}\omega_1 \\ \omega \\ \end{array}\right) 
\to \left(\begin{array}{l}\omega_1 \\ \omega \\ \end{array}\right)^{1,1}_2$
holds.
As a consequence of Theorem~\ref{tail-club}, we obtain:

\begin{Cor}
It is consistent that $\ss = \aleph_1$ and 
$\left(\begin{array}{l}\omega_1 \\ \omega \\ \end{array}\right)
\to \left(\begin{array}{l}\omega_1 \\ \omega \\ \end{array}\right)^{1,1}_2$
holds.
\end{Cor}

This answers~\cite[Question 1.7(a)]{GSta2}.

\section{Open problems}

We conclude with a number of open problems. Perhaps the most interesting is:
\begin{Question}\label{s-versus-aclosed}
Does $\ss = \aleph_1$ (or at least $\ss_\omega = \aleph_1$) imply $\aa_\closed =
\aleph_1$? 
\end{Question}
While the existence of a tail-splitting sequence of length $\omega_1$ is
strictly stronger
than the existence of a club-splitting sequence of length $\omega_1$
(Theorem~\ref{tail-club}), 
we in fact do not know whether the latter is stronger than $\ss_\omega =
\aleph_1$ or $\ss = \aleph_1$.
\begin{Question}\label{versions-of-s}
Is it consistent that $\ss = \aleph_1$ (or even $\ss_\omega = \aleph_1$) and
there is no
club-splitting sequence of length $\omega_1$?
\end{Question}
For the proof of $\aa_\closed = \aleph_1$ we needed a club-splitting sequence of
partitions
(Lemma~\ref{splitting-over-models} and
Theorem~\ref{splitting-over-models-aclosed}).
It is unclear whether a club-splitting sequence is enough. In fact, we do not
know whether
the two notions are equivalent.
\begin{Question}\label{splitting-partitions}
Does the existence of a tail-splitting sequence of length $\kappa$ imply the
existence of a
tail-splitting sequence of partitions of length $\kappa$? Similarly for
club-splitting instead of
tail-splitting.
\end{Question}
Let $\aa_\Borel$ denote the size of the smallest family $\A$ a.d. Borel sets
such that 
$\bigcup \A$ is mad. Clearly, $\aleph_1 \leq \aa_\Borel \leq \aa_\closed$.
We do not know, however, whether the cardinals are equal.
\begin{Question}[Brendle and Khomskii~{\cite[Question 4.7]{BKta}}]
\label{aBorel}
Is $\aa_\Borel = \aa_\closed$?
\end{Question}
If this is not the case one could ask
\begin{Question}[Brendle and Khomskii~{\cite[Question 4.4]{BKta}}]
\label{b-versus-aBorel}
Is $\bb < \aa_\Borel$ consistent?
\end{Question}
Finally we address
\begin{Question}[see also~{\cite[Conjecture 4.5]{BKta}}]
Is $\hh \leq \aa_\closed$? Or even $\hh \leq \aa_\Borel$?
\end{Question}



\end{document}